\documentclass{article}

\RequirePackage[top=1.7cm, bottom=1.7cm, left=2.7cm, right=2.7cm, includehead, includefoot]{geometry}
\RequirePackage[hidelinks, colorlinks=true, linkcolor=black, anchorcolor=blue, citecolor=blue, filecolor=blue, menucolor=blue, runcolor=blue, urlcolor=blue]{hyperref}

\RequirePackage{graphicx}
\RequirePackage{caption}
\RequirePackage{xspace}
\RequirePackage{color}

\RequirePackage{academicons}
\RequirePackage{xcolor}

\RequirePackage{fancyhdr}
\RequirePackage{authblk}

\RequirePackage{enumitem}
\RequirePackage{bm}
\RequirePackage{amsmath}
\RequirePackage{amsthm}
\newtheorem{theorem}{Theorem}
\newtheorem{lemma}[theorem]{Lemma}
\newtheorem{proposition}[theorem]{Proposition}
\newtheorem{corollary}[theorem]{Corollary}
\newtheorem{remark}[theorem]{Remark}

\RequirePackage{amssymb}
\RequirePackage[ruled, noend]{algorithm2e}
\newlength\lenIn
\newcommand\MyKwIn[1]{%
  \settowidth\lenIn{\KwIn{}}%
  \setlength\hangindent{\lenIn}%
  \hspace*{\lenIn}#1\\}

\newlength\lenOut
\newcommand\MyKwOut[1]{%
  \settowidth\lenOut{\KwOut{}}%
  \setlength\hangindent{\lenOut}%
  \hspace*{\lenOut}#1\\}

\RequirePackage{stmaryrd}
\RequirePackage{dsfont}

\RequirePackage{xcolor}
\definecolor{col1}{HTML}{C721DD}
\definecolor{col2}{HTML}{D14A00}
\definecolor{col3}{HTML}{008C00}
\definecolor{col4}{HTML}{007FB1}
\definecolor{col5}{HTML}{D1AC00}

\newcommand{\E}{\mathbb{E}}	
\newcommand{\N}{\mathbb{N}}	
\renewcommand{\P}{\mathbb{P}}
	
\newcommand{\R}{\mathbb{R}}	

\newcommand{\cE}{\mathcal E}
\newcommand{\cI}{\mathcal I}
\newcommand{\cC}{\mathcal C}
\newcommand{\cR}{\mathcal R}
\newcommand{\cF}{\mathcal F}
\newcommand{\cG}{\mathcal G}
\newcommand{\cN}{\mathcal N}

\newcommand{\vmu}{\bm{\mu}}
\newcommand{\vnu}{\bm{\nu}}
\newcommand{\vlambda}{\bm{\lambda}}
\newcommand{\vw}{\bm{w}}
\newcommand{\vDelta}{\bm{\Delta}}
\newcommand{\vN}{\bm{N}}
\newcommand{\vv}{\bm{v}}
\newcommand{\vu}{\bm{u}}

\newcommand{\segent}[2]{\llbracket#1, #2\rrbracket}
\newcommand{\Segent}[1]{[#1]}
\newcommand{\acco}[1]{\left \{ #1 \right \}}
\newcommand{\abs}[1]{\left \lvert #1 \right \rvert}
\newcommand{\norm}[1]{\left \lVert #1 \right \rVert}
\newcommand{\floor}[1]{\left \lfloor #1 \right \rfloor}
\newcommand{\ceil}[1]{\left \lceil #1 \right \rceil}	

\newcommand{\ind}{\mathds{1}}							

\newcommand{\ov}[1]{\overline{#1}}
\DeclareMathOperator{\card}{card}	

\DeclareMathOperator{\argmin}{argmin}
\DeclareMathOperator{\argmax}{argmax}
\DeclareMathOperator{\Alt}{Alt}
\DeclareMathOperator{\kl}{kl}

\newcommand{\TaS}{\textsc{Track-and-Stop}\xspace}
\newcommand{\EBS}{\textsc{Exploration-Biased Sampling}\xspace}
\newcommand{\CRac}{\textsc{Chernoff-Racing}\xspace}
\newcommand{\US}{\textsc{Uniform Sampling}\xspace}

\usepackage[round]{natbib}

\fancypagestyle{plain}{
	\fancyhf{}
	\fancyfoot[R]{\textbf{\thepage}}
	
	}

\begin{document}


\title{A Non-asymptotic Approach to Best-Arm Identification for Gaussian Bandits}


\author[1, 2]{Antoine Barrier}
\author[1]{Aur\'elien Garivier}
\author[3]{Tom\'a\v{s} Koc\'ak}

\affil[1]{ENS de Lyon, UMPA UMR 5669, 46 all\'ee d’Italie, 69364 Lyon Cedex 07, France}
\affil[2]{Universit\'e Paris-Saclay, CNRS, Laboratoire de mathématiques d'Orsay, 91405, Orsay, France}
\affil[3]{Institute for Mathematics, University of Potsdam, Germany}

\maketitle


\begin{abstract}
We propose a new strategy for best-arm identification with fixed confidence of Gaussian variables with bounded means and unit variance. This strategy, called \EBS, is not only asymptotically optimal: it is to the best of our knowledge the first strategy with non-asymptotic bounds that asymptotically matches the sample complexity.
But the main advantage over other algorithms like \TaS is an improved behavior regarding exploration: \EBS is biased towards exploration in a subtle but natural way that makes it more stable and interpretable. These improvements are allowed by a new analysis of the sample complexity optimization problem, which yields a faster numerical resolution scheme and several quantitative regularity results that we believe of high independent interest.

\end{abstract}


\section{Introduction}		\label{sec:intro} 

Many modern systems of automatic decisions (from recommender systems to clinical trials, through auto-ML and parameter tuning) require to find the best among a set of options, using noisy observations obtained by successive calls to a random mechanism (see e.g.~\citealp{LS20}). The simplest formal model for such situations is the \emph{standard Gaussian multi-armed bandit}, a collection of $K\geq 2$ independent Gaussian distributions called \emph{arms} of unknown means $\vmu = (\mu_a)_{1 \leq a \leq K}\in\R^K$ and variances all equal to $1$. They are sampled sequentially and independently: at every discrete 
time step $t\in\N^*$, an agent chooses an arm $A_t\in \Segent{K} = \{1, \dots, K\}$ based on past information, and observes an independent draw $Y_t$ from distribution $\cN(\mu_{A_t}, 1)$. 

Among the set $\cG$ of all standard Gaussian multi-armed bandits with means in the interval $[0,1]$, we focus in this work on the subset $\cG^*$ of bandits $\vmu \in \cG$ that have exactly one arm $a^*(\vmu)\in[K]$ with the highest mean, that is $\mu_* = \mu_{a^*(\vmu)} > \max_{a \in \Segent{K} \setminus \{ a^*(\vmu)\}} \mu_a$, and we address the problem of optimally sampling the arms in order to identify $a^*(\vmu)$ as quickly as possible. We consider the sequential statistics framework often called \emph{fixed confidence setting} (see \citealp{EMM06, KTAS12}): by defining $\cF_t = \sigma(Y_1, \dots, Y_t)$ the sigma-field generated by the observations up to time $t$, a strategy consists of a sampling rule $(A_t)_{t\geq 1}$ where each $A_t$ is $\cF_{t-1}$-measurable, a stopping rule $\tau$ with respect to $(\cF_t)_{t \geq 0}$, and a $\cF_\tau$-measurable decision rule $\hat{a}_\tau$. 
Given a risk parameter $\delta\in(0,1)$, a strategy is called \emph{$\delta$-correct} if, whatever the parameter $\vmu \in\cG^*$, it holds that $\P_{\vmu}(\tau < +\infty, \hat{a}_{\tau} \neq a^*(\vmu)) \leq \delta$. The goal is to find a $\delta$-correct strategy that minimizes the expected number of observations $\E_{\vmu}[\tau_\delta]$ needed to identify $a^*(\vmu)$. 

The sample complexity of $\delta$-correct strategies cannot be arbitrarily good: it has been proved by \cite{GK16} that they essentially obey the lower bound $\E_{\vmu}[\tau_\delta] \geq T(\vmu) \log (1/\delta)$ for any $\vmu \in \cG^*$, where the \emph{characteristic time} $T(\vmu)$ is the solution of the following optimization problem
\begin{equation}    \label{eq:def_T}
  T(\vmu)^{-1} = \sup_{\vv \in \Sigma_K} \inf_{\vlambda \in \Alt(\vmu)} \sum_{a \in \Segent{K}} v_a  \frac{(\mu_a-\lambda_a)^2}{2}\;, 
\end{equation}
where $\Sigma_K = \{\vv\in[0,1]^K : v_1+\dots+v_K = 1\}$ 
and $\Alt(\vmu) = \{ \vlambda \in \cG^* : a^*(\vlambda) \neq a^*(\vmu) \}$ is the set of bandit models with an optimal arm different from $a^*(\vmu)$. Moreover, this bound is tight: the authors introduced \TaS, a strategy for which they proved that
$\limsup_{\delta \rightarrow 0} \E_{\vmu}[\tau_\delta] / \log(1/\delta) = T(\vmu)$ (see also~\citealp{Rus16}).

The information-theoretic analysis of~\cite{GK16} also highlights the nature of the optimal sampling strategy: whatever the value of the risk $\delta$, one should sample the arms with frequencies proportional to $\vv = \vw(\vmu)$, the (unique and well-defined) maximizer in the right-hand side of Equation~\eqref{eq:def_T}.
Indeed, the \TaS algorithm works as follows: at every time step $t$, an estimate $\hat{\vmu}(t)$ of the mean parameter $\vmu$ is computed thanks to the available observations. The optimal frequencies relative to this estimate are computed, and used to determine which action is to be selected next: we pick the action that lays the most behind its estimated optimal frequency, unless one action was severely undersampled (in which case its exploration is forced). A formal description of the strategy is recalled in Appendix~\ref{appendixA} (see Algorithm~\ref{algo:Track-and-Stop}). Some improvements were proposed: for example, \cite{Men19} proved that it is not necessary to solve the optimization problem in every time step. Instead, they perform a single gradient step in every round which enables them to prove a similar result while reducing the computational complexity of their algorithm (see also \citealp{TPRL20}).

The \TaS algorithm is not only a theoretical contribution, it also proved to be numerically efficient, far exceeding its competitors in a wide variety of settings. It was improved in different directions~\citep{DK19, DKM19, SHKMV20}, and  also provides a simple template for extensions, for bandit problems with structure \citep{KG20}, as long as the optimization problem~\eqref{eq:def_T} can be solved. 
Yet, \TaS suffers from certain shortcomings. First, a close look into the proofs shows that the theoretical guarantees proved so far are really asymptotic in nature. Second, the forced exploration appears very arbitrary, with a rate of $\sqrt{t}$ that has no other justification than lying somewhere between constant and linear functions. Third, the sampling strategy appears to be pretty unstable, especially at the beginning: the target frequencies can vary significantly as the estimated means fluctuate before stabilizing around their expectations.  Fourth, \TaS does not present the intuitively desirable behavior to sample uniformly in the beginning, until sufficient information has been gathered for significant differences between the arms to emerge. This is in contrast with strategies like Racing~\citep{KK13}, which are sub-optimal but intuitively appealing. Altogether, these issues lead for example to unpredictable and irregular conduct at the beginning of multiple A/B testing cases with many arms very close to optimal. 

\paragraph*{Contributions}
The present paper addresses the issues of \TaS and proposes a new algorithm that solves all of them. 
We focus on Gaussian bandits with known and equal variances.
The exploration is conducted very differently, in a statistically natural way that softens the fluctuations of empirical means and avoids arbitrary parameters. It results in a stabilized sampling strategy, that is much easier to follow and understand. We propose for this strategy a non-asymptotic analysis with finite risk bounds.
These results have required developing a careful analysis of the quantitative regularity of the solution to the optimization problem~\eqref{eq:def_T}. As a by-product, we obtain an accelerated algorithm for its numerical resolution, allowing a significant speed-up for the \TaS or the Gradient Ascent algorithms in the Gaussian case.
Actually, the algorithms discussed here 
apply equally to sub-Gaussian arms with a known upper bound on the variances (in these settings, the sample complexity bounds proved in this paper apply but are not necessarily optimal). 

\newpage
While the proven optimality of \TaS\ is purely asymptotic, 
	a different approach is followed in~\citep{KKS13, JMNB14, CLQ17} for moderate values of $\delta$. The proposed strategies are sub-optimal by a multiplicative constant, but are proved to satisfy explicit non-asymptotic bounds. More recently, \cite{DKM19} obtained a general non-asymptotic bound, a remarkable but hardly comparable result in particular settings. In this contribution, we try to make a link between both approaches by introducing a strategy with a non-asymptotic bound that asymptotically matches the sample complexity.

The paper is organized as follows. We present in Section~\ref{sec:strategy} our new 
strategy with its main properties and guarantees. We then turn in Section~\ref{sec:sample_complexity} to the analysis of the optimization problem~\eqref{eq:def_T} and to the resulting new algorithm for its numerical resolution.
Lastly, we illustrate the performance and behavior of our strategy by numerical experiments in Section~\ref{sec:experiments}, and propose concluding remarks in Section~\ref{sec:conclusion}. 


\section{The \EBS strategy}   \label{sec:strategy}

In this section, we introduce our new strategy called \EBS. Instead of \TaS's greedy choice of actions based on a plug-in estimate of $\vmu$, it relies on a specific estimator that is biased toward uniform exploration. 

For $\vmu \in \cG$, let $\vDelta(\vmu) = (\mu_* - \mu_a)_{a \in [K]} \in [0, 1]^K$ be its \emph{gap vector} 
and $a^*(\vmu) 
= \{ a \in \Segent{K} : \Delta_a(\vmu) = 0 \}$ its set of optimal arms. When $\vmu \in \cG^*$, $a^*(\vmu)$ has one element that we also denote by $a^*(\vmu)$ and we recall that the \emph{optimal weight vector} $\vw(\vmu)$ is the unique maximizer of optimization problem~\eqref{eq:def_T}. 
Otherwise, when $\vmu \in \cG \setminus \cG^*$ has at least two optimal arms, we define $\vw(\vmu) = \frac{1}{\card(a^*(\vmu))} (\ind_{1 \in a^*(\vmu)}, \dots, \ind_{K \in a^*(\vmu)})^T$. 
Since these quantities play a special role in the sequel, we set $w_{\min}(\vmu) = \min_{a \in \Segent{K}} w_a(\vmu)$, $\Delta_{\min}(\vmu) = \min_{a \in \Segent{K} : \Delta_a(\vmu) > 0} \Delta_a(\vmu)$ (which is not defined when $a^*(\vmu)=[K]$) and $\Delta_{\max}(\vmu) = \max_{a \in [K]} \Delta_a(\vmu)$.

Given a sampling strategy, let $N_a(t) =  \sum_{s \in [t]} \ind_{A_s = a}$ be the random number of draws of arm $a \in \Segent{K}$ up to time $t\in\N^*$, and if $N_a(t)\geq 1$, let $\hat{\mu}_a(t) = N_a(t)^{-1} \sum_{s \in [t]} Y_s \ind_{A_s = a}$ be the maximum likelihood estimate of $\mu_a$ at time $t$. We use the vector notations $\vN(t) = (N_a(t))_{a \in \Segent{K}}$ and $\hat{\vmu}(t) = (\hat{\mu}_a(t))_{a \in \Segent{K}}$. 

In the rest of this section, we fix $\vmu \in \cG$.

\subsection{Conservative Tracking}

The main idea of the algorithm is to design a sampling policy of arms that naturally encourages exploration without forcing it like \TaS\ does. To do so, the objective is to ``wrap" the optimal weight vector $\vw(\vmu)$ ``from above", by ensuring that we never under-estimate its minimal value. Indeed, even an arm with low mean needs to be sampled sufficiently often until one is very confident that it is suboptimal. The idea is to construct a confidence region $\cC\cR_{\vmu} \subset [0, 1]^K$ for $\vmu$ on which one can efficiently find a bandit $\tilde{\vmu} \in \cC\cR_{\vmu}$ maximizing the minimal weight $w_{\min}$:
\begin{equation} \label{eq:optimistic_bandit_def}
	\tilde{\vmu} \in \argmax_{\vnu \in \cC \cR_{\vmu}} w_{\min}(\vnu) \;.
\end{equation}
As long as $\vmu$ belongs to the confidence region $\cC\cR_{\vmu}$, choosing the target weights $\vw(\tilde{\vmu})$ guarantees that every arm is explored sufficiently, as $w_{\min}(\tilde{\vmu}) \geq w_{\min}(\vmu)$.
The exploration bias decreases with the number of observations, as 
$\cC\cR_{\vmu}$ shrinks to $\{\vmu\}$, and in the end arms are sampled with frequencies close to the optimal weight vector $\vw(\vmu)$.

This approach to exploration requires two ingredients:
\begin{itemize}[itemsep=0cm, topsep=0cm]

	\item the exploration-biased bandit $\tilde{\vmu}$ needs to be efficiently computable. It turns out to be the case if the confidence region is a product of confidence intervals on each arm (a mild requirement since the arms are independent). We propose Algorithm~\ref{algo:OptimisticWeights}, an efficient procedure for computing $\tilde{\vmu}$. Intuitively, maximizing $w_{\min}$ over $\cC\cR(\vmu)$ requires to increase and equalize all the positive gaps as much as possible. The associated bandit will indeed be the one for which it is harder to identify the second best arm and thus it will require to sample the worst arms more frequently. This gives a candidate bandit for each potential best arm, and our algorithm compares those candidates. Figure~\ref{fig:example_possible_optimistic_bandits} illustrates on an example the principle of Algorithm~\ref{algo:OptimisticWeights}, whose correctness is proved in~Proposition~\ref{prop:optimistic_weights}. The algorithm requires {\sc Optimal Weights} (Algorithm~\ref{algo:Newton_r} of Appendix~\ref{appendixC3}), an efficient procedure for 
solving optimization problem~\eqref{eq:def_T} (see also Section~\ref{subsec:bounds_computation}). 

	\item the regularity of the mapping $\vnu \mapsto \vw(\vnu)$ needs to be explicitly known. Indeed, the confidence region will decrease with the number of observations, and $\tilde{\vmu}$ will come close to $\vmu$. The continuity proved by~\cite{GK16} for the asymptotic optimality of \TaS\ is not sufficient: the first quantitative bounds are given below in Section~\ref{subsec:regularity}.
\end{itemize}

\begin{algorithm}
	\caption{\textsc{Exploration-Biased Weights}} \label{algo:OptimisticWeights}
	\SetAlgoLined
	\DontPrintSemicolon
	\SetKwInOut{Input}{input}\SetKwInOut{Output}{output}
	\KwIn{confidence region $\cC\cR = \prod_{a \in \Segent{K}} [\underline{\mu}_a, \ov{\mu}_a]$}
	\vspace{0.1cm}
	\KwOut{exploration-biased bandit $\tilde{\vmu} \in \cC\cR$}
	\MyKwOut{exploration-biased optimal weight}
	\MyKwOut{\quad vector $\vw = \vw(\tilde{\vmu})$}

	\BlankLine

	$\text{maxLB} \gets \max_{a \in \Segent{K}} \underline{\mu}_a$ ; $\text{minUB} \gets \min_{a \in \Segent{K}} \ov{\mu}_a$ \;
	\If{$\text{minUB} \geq \text{maxLB}$} {
		$\tilde{\vmu} \gets (\text{minUB}, \dots, \text{minUB})$ ; $\vw \gets (\frac{1}{K}, \dots, \frac{1}{K})$
	}
	\Else{
	$\text{PotentialBest} \gets \{ a \in \Segent{K} ~:~ \ov{\mu}_a > \text{maxLB} \}$ \;
	$\vw \gets (0, \dots, 0)$ \;
	\For{$a \in \text{PotentialBest}$}{
	$\tilde{\mu}^{\text{test}(a)}_a \gets \ov{\mu}_a$ \;
	\For{$b \in [K] \setminus \{ a\}$}{
	$\tilde{\mu}^{\text{test}(a)}_b \gets \max(\underline{\mu}_b, \text{minUB})$
	}
	$\vw^{\text{test}(a)} \gets \text{{\sc Optimal Weights}}(\tilde{\vmu}^{\text{test}(a)})$ \;
	\If{$\min_{b \in [K]} w^{\text{test}(a)}_b > \min_{b \in [K]} w_b$}{
	$\vw \gets \vw^{\text{test}(a)}$ ; $\tilde{\vmu} \gets \tilde{\vmu}^{\text{test}(a)}$
	}
	}
	}
\end{algorithm}

\begin{figure*}
	\centering
	\includegraphics[width=0.32\linewidth]{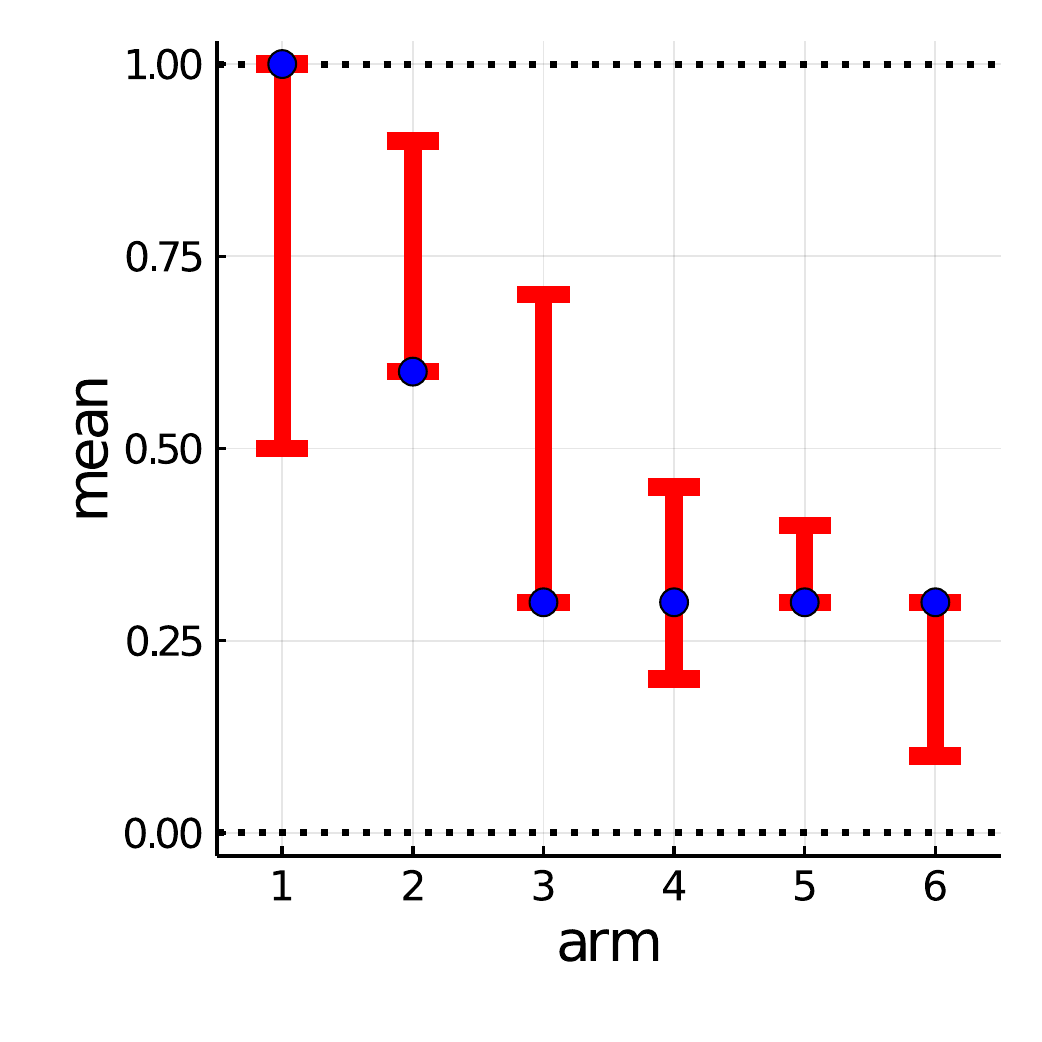}
	\includegraphics[width=0.32\linewidth]{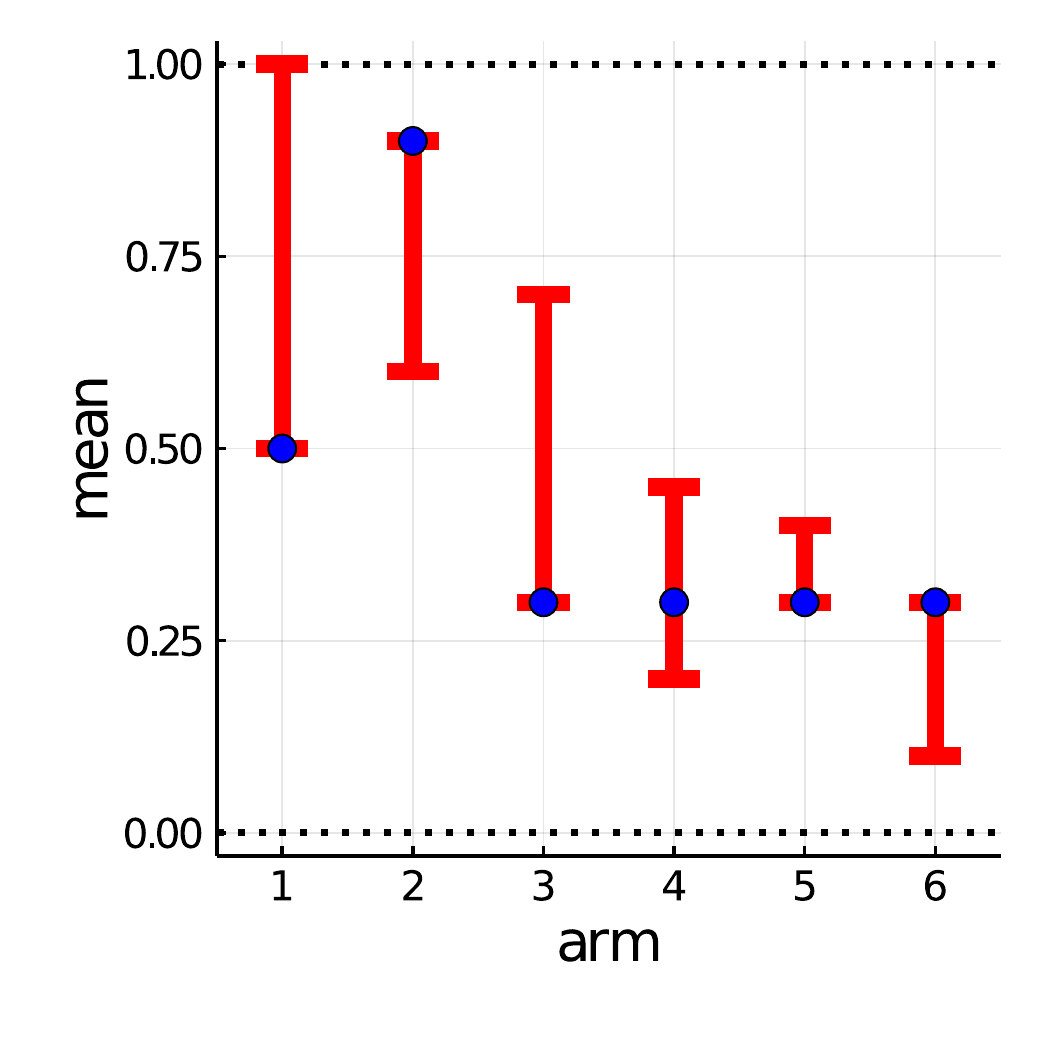}
	\includegraphics[width=0.32\linewidth]{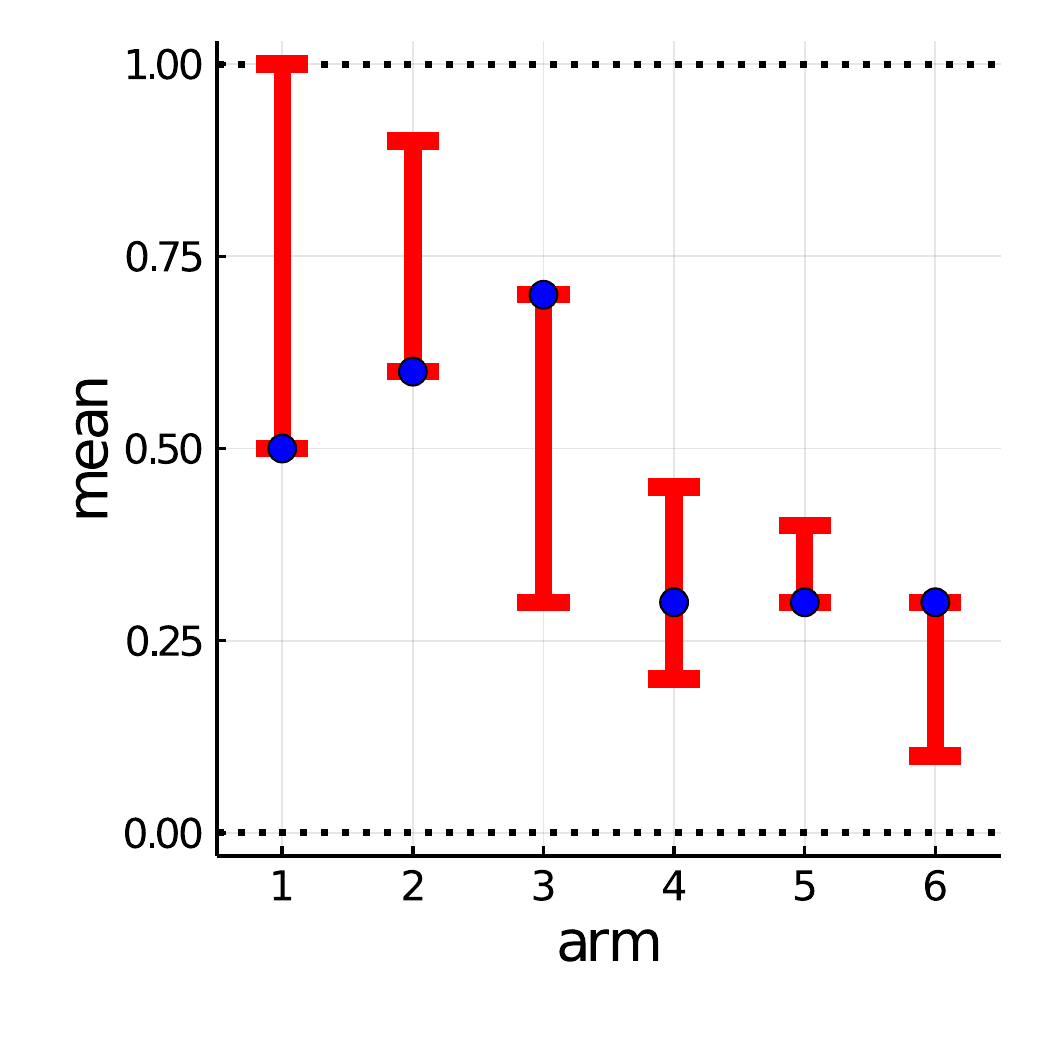}
	\caption{List of bandits $(\tilde{\vmu}^{\text{test}(a)})_{a \in \text{PotentialBest}}$ tried by Algorithm~\ref{algo:OptimisticWeights} for the example confidence region in red with $\text{PotentialBest} = \{1, 2, 3\}$. From left to right: $\tilde{\vmu}^{\text{test}(1)}$, $\tilde{\vmu}^{\text{test}(2)}$ and $\tilde{\vmu}^{\text{test}(3)}$ 	\label{fig:example_possible_optimistic_bandits}}
\end{figure*}

One can remark that as long as the confidence intervals have a non-empty intersection, which means the observations do not permit to exclude that any of them is optimal, the exploration-biased weights returned by Algorithm~\ref{algo:OptimisticWeights} are uniform and the arms are sampled in a round-robin way (as in a Racing or Successive Elimination algorithm like in~\citep{EMM06}).

\begin{proposition} \label{prop:optimistic_weights}
	Let $\cC\cR = \prod_{a \in \Segent{K}} [\underline{\mu}_a, \ov{\mu}_a] \subset [0, 1]^K$ and $(\tilde{\vmu}, \vw) \gets \text{\textsc{Exploration-Biased Weights}}(\cC\cR)$. Then $\vw = \vw(\tilde{\vmu})$ and $\tilde{\vmu}$ satisfies Equation~\eqref{eq:optimistic_bandit_def}. 
\end{proposition}


The proof of Proposition~\ref{prop:optimistic_weights} is given in Appendix~\ref{appendixC4} and relies on the results of Section~\ref{subsec:monotonicity_minmaxpb}.

\subsection{The Strategy}

We are now able to introduce our strategy called \EBS. Given a risk $\delta \in (0, 1)$ and a threshold function $\beta(t, \delta)$, we compute at each time confidence intervals for each $\mu_a$ that will ensure $\vmu$ to belong to each associated confidence region with probability at least $1-\gamma$, where $\gamma \in (0, 1)$ is a fixed parameter. We can then ensure enough exploration by biasing the optimal weights $\vw(\vmu)$ using Algorithm~\ref{algo:OptimisticWeights}. 


\paragraph*{Confidence regions} Confidence regions are designed to satisfy two requirements. First we need products of confidence intervals in order to use Algorithm~\ref{algo:OptimisticWeights}, and then we will require a time-uniform confidence guarantee as a key ingredient for the non-asymptotic analysis of \EBS. For $\gamma \in (0, 1)$, we define for $t \in \segent{K}{\tau_\delta}$
\begin{equation} \label{eq:def_CR}
	\textstyle \cC\cR_{\vmu}(t) = \prod_{a \in \Segent{K}} \bigl[\hat{\mu}_a(t) \pm C_{\gamma/K}(N_a(t)) \bigr] \;,
\end{equation}
where $C_\gamma(s) = 2\sqrt{\frac{\log(4s/\gamma)}{s}}$. The following Lemma, proved in Appendix~\ref{appendixB}, states a time-uniform $\gamma$-confidence guarantee for $\vmu$.

\begin{lemma} \label{lem:CR_mu} For any $\vmu \in \cG$ and $\gamma \in ]0, 1[$, we have
	\[ \P_{\vmu}\bigl(\exists t \in \segent{K}{\tau_\delta} : \vmu \notin \cC\cR_{\vmu}(t) \bigr) \leq \gamma \;. \]
\end{lemma}


\paragraph*{Stopping rule} Following~\cite{GK16}, our stopping rule relies on the statistic
\[ Z(t) = \max_{a \in \Segent{K}} \min_{b \neq a} 
	Z_{a, b} (t) \;,\]
where $Z_{a, b}(t)$ is the Generalized Likelihood Ratio statistic (see~\citealp{Che59}), equal in the Gaussian case to
\[ Z_{a, b} (t) = \frac{1}{2} \frac{N_a(t) N_b(t)}{N_a(t) + N_b(t)} (\hat{\mu}_a (t) - \hat{\mu}_b (t)) \big|\hat{\mu}_a (t) - \hat{\mu}_b (t)\big|\;. \]

\begin{algorithm}
	\caption{\EBS \label{algo:ExplorationBiased}}

	\DontPrintSemicolon
	\SetKwInOut{Input}{input}\SetKwInOut{Output}{output}
	\KwIn{confidence level $\delta$}
	\MyKwIn{threshold function $\beta(t, \delta)$}
	\MyKwIn{confidence parameter $\gamma$}
	\KwOut{stopping time $\tau_\delta$}
	\MyKwOut{estimated best arm $\hat{a}_{\tau_\delta}$}

	\BlankLine
	Observe each arm once ; $t \gets K$ \;
	\For{$s = 0$ \KwTo $K-1$} {
		$\tilde{\vw}(s) \gets (1/K, \dots, 1/K)$
	}
	\While{$Z(t) \leq \beta(t, \delta)$}{
		$\cC\cR_{\vmu}(t) \gets \prod_{a \in \Segent{K}} [\hat{\mu}_a(t) \pm C_{\gamma/K} (N_a(t))]$\;
		$(\tilde{\vmu}(t), \tilde{\vw}(t)) \gets $ {\sc Exploration-Biased Weights}($\cC\cR_{\vmu}(t)$)\;
		Choose $A_{t+1} \in \argmin_{a \in \Segent{K}} N_a(t) - \underset{s \in [t]}{\sum} \tilde{w}_a(s)$ \;	
		Observe $Y_{A_{t+1}}$ and increase $t$ by $1$
	}
	$\tau_\delta \gets t$ ; $\hat{a}_{\tau_\delta} \gets \argmax_{a \in \Segent{K}} \hat{\mu}_a (t)$
\end{algorithm}

The \EBS strategy is summarized in Algorithm~\ref{algo:ExplorationBiased}. As explained in~\cite{GK16}, one can either follow the exploration-biased weights directly (D-tracking) or their cumulative sums (C-tracking). For the simplicity of the proofs, we use C-tracking in the analysis, but we ran the experiments with both options, as D-tracking appears to perform slightly better (replace $\sum_{s \in [t]} \tilde{w}_a(s)$ by $t\tilde{w}_a(t)$ in the description of Algorithm~\ref{algo:ExplorationBiased} for D-tracking). 

It happens that the choice of confidence regions given by Equation~\eqref{eq:def_CR} leads to a minimal exploration rate for each arm of order $\sqrt{t}$. What is surprising is that this is exactly the arbitrary rate used by \TaS for forced exploration, which appears here naturally. 

\begin{lemma} \label{lem:LB_exploration_rate} For any choice of parameters and $\vmu \in \cG$, \EBS satisfies
\[ \forall t \in \segent{0}{\tau_\delta}, \forall a \in \Segent{K}, \quad N_a(t) \geq \frac{2}{K} \sqrt{t} - K\;. \]
\end{lemma}
The proof of this lemma can be found in Appendix~\ref{appendixF1}.

The practical advantages of \EBS over \TaS are discussed in Section~\ref{sec:experiments}. On the theoretical level, we now show that (contrary to \TaS) this exploration strategy is adequate for obtaining non-asymptotic bounds.

\subsection{Theoretical Results}

\paragraph*{A $\delta$-correct strategy} The  $\delta$-correctness of our strategy, which relies on the same stopping rule as \TaS, is a simple consequence of~\citet[Proposition 12]{GK16}.
\begin{proposition} \label{prop:delta_correct} For any $\delta, \gamma \in (0, 1)$ and $\alpha > 1$, there exists a constant $R = R(K, \alpha)$ such that \EBS with parameters $\delta, \gamma$ and threshold 
	\begin{equation} \label{eq:threshold}
		\beta(t, \delta) = \log \Bigl(\frac{Rt^\alpha}{\delta} \Bigr)
	\end{equation}
	is $\delta$-correct.
\end{proposition}

Our main result is to obtain high probability bounds for $\tau_\delta$ in finite horizon for \EBS, which is summarized in the following theorem. 

\begin{theorem}[Non-asymptotic bound] \label{thm:bound} Fix $\gamma \in (0, 1)$, $\alpha \in [1, 2]$, $\eta \in (0, 1]$ and let $\vmu \in \cG^*$. There exists an event $\cE$ of probability at least $1-\gamma$ and $\delta_0 = \delta_0(\vmu, K, \gamma, \eta, \alpha) > 0$ such that for any $0 < \delta \leq \delta_0$, algorithm \EBS with the threshold of Equation~\eqref{eq:threshold} satisfies
	\begin{equation} \label{eq:bound_tail}
		\P_{\vmu}\bigl( \tau_\delta > t ~\cap~ \cE \bigr) \leq 2Kt\exp \Bigl( - \frac{tw_{\min}(\vmu)}{4{T(\vmu)}^2} \frac{1}{\log^{\frac{2}{3}}(1/\delta)} \Bigr)
	\end{equation}
	for any $t > (1+\eta) T(\vmu) \log(1/\delta)$, and 
	\begin{equation} \label{eq:bound_E_tau}
	\E_{\vmu}[\tau_\delta \ind_\cE] \leq (1+\eta) T(\vmu)\log(1/\delta) + \frac{2^7 K {T(\vmu)}^4}{{w_{\min}(\vmu)}^2} \exp\Bigl( -\frac{w_{\min}(\vmu)}{4T(\vmu)} \log^{\frac{1}{3}}(1/\delta)\Bigr) \log^2(1/\delta) \;. 
	\end{equation}
\end{theorem}

\newpage
Note that: 
\begin{itemize}[itemsep=0cm, topsep=0cm]
	\item using the results of Section~\ref{sec:sample_complexity}, one can show that $w_{\min}(\vmu) \geq \frac{\Delta_{\min}(\vmu)}{2K}$ for any $\vmu \in \cG^*$ (see Lemma~\ref{lem:bound w_min} in Appendix~\ref{appendixF1}), 

	\item the proof of Theorem~\ref{thm:bound} provides an explicit expression for $\delta_0$, 
	
	\item the second term of Bound~\eqref{eq:bound_E_tau} tends to $0$ when $\delta$ decreases to $0$, and hence negligible with respect to the first term: the sample complexity is therefore arbitrarily close to the lower bound. 
\end{itemize}

We additionally prove that, from an asymptotic point of view, the \EBS algorithm presents the same guarantees as \TaS (see also Theorem~\ref{thm:asymptotic_optimality_as} in Appendix~\ref{appendixF2}):
\begin{theorem}[Asymptotic optimality in expectation] \label{thm:asymptotic_optimality_E} Fix $\gamma \in (0, 1)$, $\alpha \in (1, e/2]$ and let $\vmu \in \cG^*$. Algorithm \EBS with the threshold of Equation~\eqref{eq:threshold} satisfies
\[ \textstyle \limsup_{\delta \to 0} \frac{\E_{\vmu}[\tau_{\delta}]}{\log(1/\delta)} \leq \alpha T(\vmu) \;. \]
\end{theorem}
Appendix~\ref{appendixD} will be devoted to the proof of Theorem~\ref{thm:bound} while the proof of Theorem 
\ref{thm:asymptotic_optimality_E} can be found in Appendix~\ref{appendixF3}. 

It is worth mentioning that the guarantees of \EBS presented in this section hold true not only for Gaussian arms, but more generally for $1$-sub-Gaussian arms with means in $[0, 1]$ (in which case, of course, a better lower bound might hold); indeed, these proofs only rely on sub-Gaussian deviation bounds.


\section{About the sample complexity optimization problem} \label{sec:sample_complexity}

We now introduce a new method for solving the sample complexity optimization problem~\eqref{eq:def_T}. It comes with a new analysis that yields various bounds for the bandits characteristic constants together with monotonicity and regularity results. Detailed discussions and proofs are deferred to Appendix~\ref{appendixC}. 

In this section, 
letters $a, b, c$ always refer to arm indices, that is elements of $\Segent{K}$. In subindices for sums and infima, we sometimes omit to explicitly mention $\Segent{K}$ for simplicity: for example, given a fixed arm $b$, $\sum_{a \neq b}$ denotes the sum over arms $a \in \Segent{K} \setminus \{b\}$.

For any bandit $\vmu \in \cG$ and $\vv \in \Sigma_K$, we define:
\begin{align}  g(\vmu, \vv) &
	= \inf_{\vlambda \in \Alt(\vmu)} \sum_{a \in \Segent{K}} v_a \frac{(\mu_a-\lambda_a)^2}{2} \\
	&=  \frac{1}{2} \min_{a \neq a^*} \frac{v_{a^*}v_a}{v_{a^*}+v_a} \Delta_a(\vmu)^2 \;. \label{eq:def_g}
\end{align}
The easy proof of the second equality can be found in Appendix~\ref{appendixC1}. 
Function $g$ is twice useful, as the solution to the inner optimization problem~\eqref{eq:def_T}, and for the expression of the statistic $Z(t)$:
\begin{align}
	T(\vmu)^{-1}                    & = g(\vmu, \vw(\vmu)) \label{eq:link_T_g} \;,                         \\
	\text{and} \quad \quad \quad Z(t) & = t \,g \Big( \hat{\vmu} (t), \frac{\vN(t)}{t} \Big) \label{eq:link_Z_g}
\end{align}
with the convention $T(\vmu) = +\infty$ when $\vmu \in \cG \setminus \cG^*$.

Let in this section $\vmu \in \cG^*$ be a fixed bandit parameter. For the simplicity of the presentation, let $a^* = a^*(\vmu)$, $\vDelta = \vDelta(\vmu)$, $\vw = \vw(\vmu)$, $w_{\min} = w_{\min}(\vmu)$ and $T = T(\vmu)$.

\subsection{Solving the Optimization Problem} \label{subsec:optimization_problem}

We define
\begin{equation} \label{eq:def_phi}
	\phi_{\vmu}: r \in \Big(\frac{1}{\Delta_{\min}^2}, +\infty \Big) \longmapsto \sum_{a \neq a^*} \frac{1}{(r\Delta_a^2-1)^2} - 1 \;.
\end{equation}

\begin{lemma} \label{lem:convexity_phi} $\phi_{\vmu}$ is convex and strictly decreasing on $(1/\Delta_{\min}^2, +\infty)$, and thus has a unique root. 
\end{lemma}
The following proposition shows that solving $\phi_{\vmu}(r) = 0$ directly gives a solution to Problem~\eqref{eq:def_T}.
\begin{proposition} \label{prop:links_r} 
	Let $r = r(\vmu)$ be the solution of $\phi_{\vmu}(r) = 0$. Then
	\begin{align}
		w_{a^*}                       & = \frac{1}{1+\sum_{a \neq a^*} \frac{1}{r\Delta_a^2 - 1}} \;, 	\label{eq:link_r_w1} \\
		\forall a \neq a^*, \quad w_a & = \frac{w_{a^*}}{r\Delta_a^2 - 1} \;, \label{eq:def_wa}                          \\
		\hbox{and \qquad} T           & = 2 \frac{r}{w_{a^*}} \;. 	\label{eq:link_r_T} 
	\end{align}
	Besides,
	\begin{equation}
		w_{a^*} = \sqrt{\sum_{a \neq a^*} {w_a}^2} \;.	\qquad \qquad ~ \label{eq:w1=norm2_wa}
	\end{equation}
\end{proposition}
Recall that in the case of $2$ arms, $\vw(\vmu) = (0.5, 0.5)$. Besides, the monotonicity of the optimal weights with respect to the gaps  follows from Equation~\eqref{eq:def_wa}.
\begin{corollary} \label{cor:weights_order} Assume that $K \geq 3$. 
Then
	\[ \forall a, b \in \Segent{K}, \quad \mu_a > \mu_b \quad \Longrightarrow \quad w_a > w_b \;. \]
\end{corollary}
Equation~\eqref{eq:def_wa} also implies that
\[ \forall a, b \neq a^*, \quad \frac{w_a}{w_b} = \frac{\Delta_b^2 - 1/r}{\Delta_a^2 - 1/r} \;. \]
Intuitively, it requires about $\Delta_a^2$ samplings of arms $a^*$ and $a$ before being able to distinguish them, so that one could expect $\frac{w_a}{w_b}$ to be $\frac{\Delta_b^2}{\Delta_a^2}$. This would be the case if the comparisons between arms were independent. In our problem, sampling the best arm benefits the comparison with all arms, so that it is worth sampling the optimal arm a little more than any single comparison would require, and hence each sub-optimal arm a little less. As a result, the ratio $\frac{w_a}{w_b}$ is closer to $1$, and the factor can be seen as a ``discount" on each squared gap for sharing the comparisons. We now derive other important consequences of Proposition~\ref{prop:links_r}.

\subsection{Bounds and Computation of the Problem Characteristics}	\label{subsec:bounds_computation}

By Proposition~\ref{prop:links_r}, it suffices to compute $r$ to obtain the values of both $T$ and $\vw$. As $\phi_{\vmu}$ is a strictly convex and strictly decreasing function, Newton's iterates initialized with a value $r_0 < r$ converge to $r$ from below  at quadratic speed. The procedure is summarized in Algorithm~\ref{algo:Newton_r} of Appendix~\ref{appendixC3}. The number of correct digits roughly doubles at every step, which implies that a few iterations are sufficient to guarantee machine precision. The cost of the algorithm can hence be considered proportional to that of evaluating $\phi_{\vmu}(r)$, which is linear in the number of arms.

It remains to show that it is possible to find $r_0<r$, and possibly close to $r$. The next proposition offers such a lower bound as simple functions of the gaps. This also yields tight bounds on the optimal weight vector $\vw$ and the characteristic time $T$.

\begin{proposition} \label{prop:bounds} 
	Denoting by $\overline{\Delta^2} = \frac{1}{K-1} \sum_{a \neq a^*} \Delta_a^2$ the average squared gap, 
	\begin{align}
		\max\left(\frac{2}{\Delta_{\min}^2} , \frac{1+\sqrt{K-1}}{\overline{\Delta^2}}\right)                      & \leq r \leq \frac{1+\sqrt{K-1} }{\Delta_{\min}^2}\;,	\label{eq:bounds_r}                   \\
		\frac{1}{1+\sqrt{K-1}}                                                                                     & \leq w_{\max} \leq \frac{1}{2}\;,	\label{eq:bounds_w1}                                   \\
		\max\left(\frac{8}{\Delta_{\min}^2}, 4 \frac{1+\sqrt{K-1}}{\overline{\Delta^2}} \right) & \leq T \leq 2 \frac{\big(1+\sqrt{K-1}\big)^2 }{\Delta_{\min}^2}	\label{eq:bounds_T} \;.
	\end{align}
\end{proposition}
Note that all of these inequalities can be reached for certain parameters $\vmu$, as discussed in Appendix~\ref{sec:proofProp4} after the proof of Proposition~\ref{prop:bounds}.

\subsection{Monotonicity of the $\min$-$\max$ Problem}	\label{subsec:monotonicity_minmaxpb}

We now show monotonicity results of the mappings $\vnu \mapsto T(\vnu)$ and $\vnu \mapsto \vw(\vnu)$ when moving arm(s). When $K = 2$, the optimization problem is simple and leads to $\vw(\vmu) = (0.5, 0.5)$ and $T(\vmu) = 8\Delta_2^2$, so that we assume in the remaining of this section that $K \geq 3$. 

Let $\vmu' \in \cG^*$ be another bandit problem sharing the same unique optimal arm $a^*$ as $\vmu$ and define $\vDelta'$, $\vw'$, $w_{\min}'$, $T'$ and $r'$ similarly to problem $\vmu$. 
The three following lemmas, which are the key ingredients to prove Proposition~\ref{prop:optimistic_weights}, are shown in Appendix~\ref{appendixC4}. 
\begin{lemma} \label{lem:moving_one_suboptimal_arm} Assume that $\Delta'_b > \Delta_b$ for a fixed $b \neq a^*$ while $\Delta'_a = \Delta_a$ for all 
	$a\neq b$. Then
	\begin{enumerate}[itemsep=0cm, topsep=0cm]
		\item \label{lem:moving_one_suboptimal_arm1} $w'_b < w_b$,
		\item \label{lem:moving_one_suboptimal_arm2} $w'_a > w_a$ for any $a \notin \{a^*, b\}$,
		\item \label{lem:moving_one_suboptimal_arm3} $T' < T$.
	\end{enumerate}
\end{lemma}

\begin{lemma} \label{lem:moving_optimal_arm} Assume that $\Delta'_a = \Delta_a + d$ for every $a \neq a^*$ and some $d > 0$. Then $w'_{\min} \geq w_{\min}$, with strict inequality whenever $\Delta_a \neq \Delta_b$ for some $a, b \neq a^*$. \end{lemma}

\begin{lemma} \label{lem:moving_worst_arms} Let $B = \argmin_{a \in \Segent{K}} \mu_a$ (resp. $B'= \argmin_{a \in \Segent{K}} \mu'_a$) be the set of the worst arms of $\vmu$ (resp. $\vmu'$) and assume that $B \subset B'$ and $\Delta'_{\max} < \Delta_{\max}$, while $\Delta'_a = \Delta_a$ for all $a \notin B'$. Then $w'_{\min} \geq w_{\min}$.
\end{lemma}


\subsection{Regularity of \texorpdfstring{$\vw$} ,, $T$ and $g$}		\label{subsec:regularity}

Lastly, we show explicit bounds on the regularity of $\vnu \mapsto \vw(\vnu)$ and $\vnu \mapsto T(\vnu)$. We keep the notations of the last section. 
\begin{theorem}\label{thm:regularity} 
	Assume that $(1-\varepsilon) \Delta_a^2 \leq {\Delta_a'}^{2} \leq (1+\varepsilon) \Delta_a^2$ for all $a \neq a^*$ and some $\varepsilon \in [0, 1/7]$. Then
	\begin{align*}
		(1-3\varepsilon) T                                   & \leq  T' \leq (1+6\varepsilon)T \;,       \\
		\forall a \in \Segent{K}, \quad (1-10\varepsilon)w_a & \leq   w'_a \leq (1+10\varepsilon)w_a \;.
	\end{align*}
\end{theorem}
Independently, we show the following property of $g$.

\begin{proposition} \label{prop:control_g}
	Let $\vv \in \Sigma_K$. Then:
	\[ g(\vmu', \vv) \geq \frac{(1-\eta)^2}{1+\eta} \big( g(\vmu, \vw(\vmu)) - \varepsilon/2 \big) \]
	where $\varepsilon = \norm{\vmu-\vmu'}_\infty$ and $\eta = \max_{a \in \Segent{K}} \frac{\abs{w_a(\vmu) - v_a}}{w_a(\vmu)}$.
\end{proposition}
These results will prove to be essential to the proof of the non-asymptotic bounds of Theorem~\ref{thm:bound}. 


\section{Numerical experiments}	\label{sec:experiments}

In this section, we discuss the behavior and performance of \EBS for practical values of confidence $\delta$. We propose a comparison with \TaS, \CRac and LUCB++, and begin with a reminder on those strategies. 

\noindent \TaS \quad	The strategy tracks the optimal weights $\vw(\vmu)$ by estimating it by $\vw(\hat{\vmu}(t))$. Some exploration rate is forced to ensure that bad initial observations does not lead to an under-sampling of some arms (the strategy ensures that each $N_a(t)$ growths at least in $\sqrt{t}$). The stopping rule is the same as the one presented for \EBS. 

\noindent\CRac \quad	The strategy is divided into rounds during which the arms of a currently active set are sampled once. At the end of each round, a decision is made to keep or eliminate the current worst arm from the active set. Several decision rules are possible, we will use the Chernoff rule presented in~\citep{GK16}, which eliminates arm $b$ at the end of round $r$ if $Z_{\hat{a}_r, b}(t) = \frac{r}{4} (\hat{\mu}_{\hat{a}_r}(t) - \hat{\mu}_b(t))^2 > \beta(t, \delta)$ where $\hat{a}_r$ (resp. $t$) is the best arm (resp. the time) at the end of round $r$.


\noindent LUCB++ \quad	The strategy~\citep{SJR17} (see also \citealp{KTAS12, HRMS21}) samples two arms at each round: the one with the current best estimate and the one in the remaining arms with the highest optimistic indice $U_a(t)$ which is an upper confidence bound: 
\[ U_a(t) = \hat{\mu}_a(t) + \sqrt{\frac{3}{N_a(t)} \log\Bigl( \frac{\log(N_a(t)) \times 2K}{\delta}\Bigr)} \]
(constant $\sqrt{3}$ appeared to be empirically optimal). For the fairness of the comparison we will take the same stopping condition as \TaS and \EBS. 

\noindent \EBS \quad We ran our experiments with confidence lengths $C_\gamma(s) = \sqrt{\frac{\log(s/\gamma)}{s}}$, and for all strategies we used the same threshold \[ \beta(t, \delta) = \log((\log(t)+1)/\delta)\;. \] 
These choices are more aggressive than what the theoretical analysis suggests: yet, empirically, they appears to guarantee  the desired failure rate. Using the larger intervals of Section 2 would have increased the number of rounds with uniform exploration, and using larger thresholds unnecessarily delays the stopping for all strategies.

We now discuss the numerical pros and cons of \EBS. 

\paragraph*{Improving the Stability of \TaS} In Section~\ref{sec:intro}, we highlighted the weaknesses of \TaS, especially the forced exploration parameter and the non-interpretable and unstable sampling strategy during the first rounds. 
On Figures~\ref{fig:experiment1} and~\ref{fig:experiment1_zoom} we see the improvements of \EBS concerning those behaviours. During the first rounds, as for a racing algorithm, a uniform sampling is observed as the learner has not collected enough information (the confidence intervals on all arms are not separated), which is the expected behavior. Then the best arms are sampled more and more often, but still in a more cautious way than \TaS. We observe on Figure~\ref{fig:experiment1_zoom} the stability of the sampling strategies comparing to \TaS during the first rounds: the targeted weights of \EBS are stable and separate from each other cautiously (note that the three last arms still have the same weight at time $1200$) whereas for \TaS, we observe an important variation of the targeted weights with time. As a matter of facts, there is a clear discontinuity each time the estimated best arm changes, as we can see with the red and green arms. We also remark that \TaS uses forced exploration at regular rounds (giving the yellow and blue peaks), which is unnecessary for \EBS as a natural exploration is always performed (Lemma~\ref{lem:LB_exploration_rate}). 

\begin{figure}[h]
  \centering
  \includegraphics[width=0.49\textwidth]{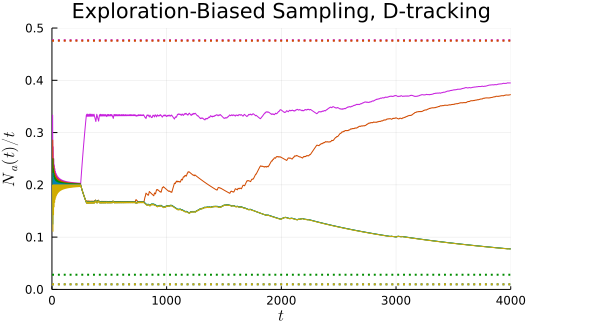} \includegraphics[width=0.49\textwidth]{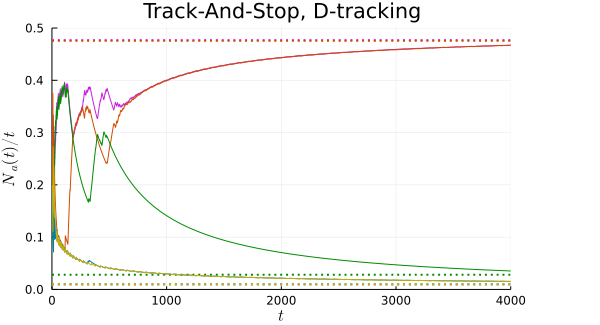}
  \caption{Evolution of the Sampling Frequencies $\vN(t)/t$ on a Simulation of {\sc Exploration-Biased Sampling} and \TaS. ($\delta = 0.01$, $\gamma = 0.2$, and $\vmu = (\textcolor{col1}{0.9}, \textcolor{col2}{0.8}, \textcolor{col3}{0.6}, \textcolor{col4}{0.4}, \textcolor{col5}{0.4})$; the values of $\vw(\vmu) = (\textcolor{col1}{0.477}, \textcolor{col2}{0.476}, \textcolor{col3}{0.028}, \textcolor{col4}{0.010}, \textcolor{col5}{0.010})$ are dotted)} \label{fig:experiment1}
\end{figure}

\begin{figure}[h]
  \centering
  \includegraphics[width=0.49\textwidth]{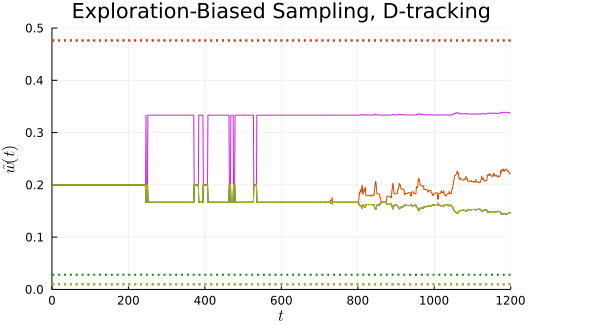} \includegraphics[width=0.49\textwidth]{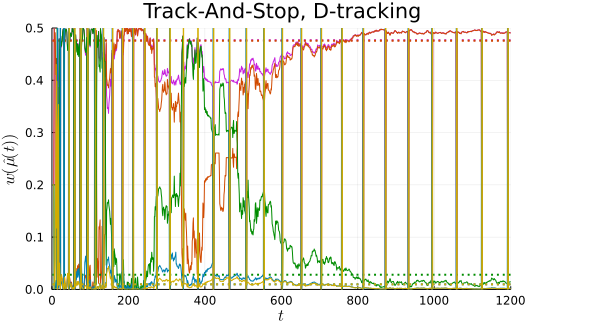}
  \caption{Evolution of the Targeted Weights $\tilde{\vw}(t)$ (resp. $\vw(\hat{\mu}(t))$) During the First $1200$ Rounds on a Simulation of \EBS (resp. \TaS). ($\delta = 0.01$, $\gamma = 0.2$, $\vmu = (\textcolor{col1}{0.9}, \textcolor{col2}{0.8}, \textcolor{col3}{0.6}, \textcolor{col4}{0.4}, \textcolor{col5}{0.4})$)} \label{fig:experiment1_zoom}
\end{figure}

\begin{table*}[h]
	\caption{Empirical Expected Number of Draws $\E_{\vmu}[\tau_\delta]$, Averaged over $1000$ Experiments: $\vmu^{(1)} = (0.9, 0.8, 0.6, 0.4, 0.4)$, $\vw(\vmu^{(1)}) = (0.477, 0.476, 0.028, 0.010, 0.010)$; $\vmu^{(2)} = (0.9, 0.5, 0.45, 0.4)$, $\vw(\vmu^{(2)}) = (0.375, 0.286, 0.195, 0.144)$ }
	\begin{center}
		\begin{tabular}{|c|c|c|c||c|c||c|c||c||c|}
			\hline
			Bandit       & $\delta$  & $\gamma$ & $T \kl(\delta, 1-\delta)$ & EBS C & TaS C & EBS D & TaS D & Racing & LUCB++ \\ \hline\hline
			$\vmu^{(1)}$ & 0.1       & 0.05     & 1476                        & 4727     & 3597  & 4191     & 3477  & 3124 & 3353   \\ \hline
			$\vmu^{(1)}$ & 0.01      & 0.05     & 3782                        & 7363     & 5664  & 6330     & 5584  & 5419 & 5549  \\ \hline
			$\vmu^{(1)}$ & 0.01      & 0.2      & 3782                        & 7090     & 5664  & 6136     & 5584  & 5419 & 5372 \\ \hline 
			$\vmu^{(1)}$ & $10^{-5}$ & 0.2      & 9669                        & 13801    & 12181 & 12376    & 11439 & 11557 & 11644  \\ \hline \hline
			$\vmu^{(2)}$ & 0.1       & 0.05     & 135                         & 476      & 367   & 470      & 322   & 405 & 365   \\ \hline
			$\vmu^{(2)}$ & 0.01      & 0.05     & 347                         & 708      & 588   & 699      & 485   & 542 & 565    \\ \hline
		\end{tabular}
	\end{center}
	
	\label{tab:experiments}
\end{table*}

\paragraph*{Comparisons of the Strategies} The cost of the the cautiousness of the algorithm (the exploration-biased weights) is that it takes a little longer for the proportions of draws of \EBS to converge to the optimal weights. This results in a slightly larger stopping time than \TaS that occurs for every bandit parameter\footnote{Note that the cautiousness of our strategy is required to obtain the non-asymptotic bounds of Theorem~\ref{thm:bound}. 
}. This can be observed on Table~\ref{tab:experiments}, where we present the performances of \EBS, \TaS, \CRac and LUCB++ with two scenarios and a set of parameters. \EBS globally performs correctly but we see that the other strategies are always a little more efficient. Note that when increasing $\gamma$, the confidence intervals reduces so that the targeted weights are closer to $\vw$, improving the performance of the algorithm. For similar reasons the initial cautiousness of the strategy disappears at long-term, thus when $\delta$ is very small the relative performance of \TaS and \EBS gets closer. Of course, \EBS overperforms \CRac in the long run when the optimal weights are far from the sampling proportions of \CRac (e.g. when $w_1\gg w_2$). 

\CRac shows great performance with both $\vmu^{(1)}$ and $\vmu^{(2)}$. This strategy samples the two last arms of the race equally often, thus can be optimal only when $\vw(\vmu)$ has its two highest components of similar value, e.g. when the two best arms are well separated from the others : this is the case of bandit $\vmu^{(1)}$. For $\vmu^{(2)}$ any strategy performs well as the problem is easy. However, \CRac (whose theoretical analysis remains to be written) leads to a few more misidentifications in our experiments that might be linked to the stopping rule we chose here; for fairness reasons, it was taken identical to that of the other algorithms. LUCB++ presents similar performance with \CRac, which can be explained by the similar behaviour of the strategies: LUCB++ samples half time the best arm asymptotically, and the worst arms are eliminated one by one once their indice fall under the two best estimates. 

Finally, note that D-tracking shows better performance than C-tracking, either for \EBS and \TaS. D-tracking indeed benefits directly of the current estimate of $\vmu$ (thus the empirical proportions of draws converge faster to the optimal weight), while the impact is diluted in time with C-tracking. However we did not prove theoretical guarantees for D-tracking. 

Additional experiments showing and interpreting the dependence on parameter $\delta$ of \EBS are postponed to Appendix \ref{appendixG}.


\section{Conclusion}\label{sec:conclusion}

We introduced \EBS, a new strategy for the problem of best arm identification with fixed confidence. In addition to asymptotic optimal results, we proved non-asymptotic bounds for this strategy in the case of (sub-)Gaussian bandits. Those finite risk bounds were made possible by a new analysis of the sample complexity optimization problem, and by the design of our strategy which tackles the shortcomings of \TaS: the procedure ensures exploration in an unforced way and stabilizes the sampling strategy, observing uniformly before having a high certainty that one arm is better than another.

It would be interesting but it remains out of reach to generalize this approach to non-Gaussian models:  this requires to extend our results on the sample-complexity optimization problem,  technically challenging task for which the simple and clean arguments developed here are likely to be replaced by much more involved derivations, if this is possible. In addition, it will be necessary to modify the confidence intervals on the arm means in a way that ensures exploration. Another direction of improvement will be to investigate if similar analysis and strategies are possible for the problem of $\varepsilon$-best arm identification.


\subsubsection*{Acknowledgements}
Aurélien Garivier and Antoine Barrier acknowledges the support of the Project IDEXLYON of the University of Lyon, in the framework of the Programme Investissements d'Avenir (ANR-16-IDEX-0005), and Chaire SeqALO (ANR-20-CHIA-0020-01).


\bibliographystyle{apalike}
\bibliography{biblio}


\newpage
\appendix

\section*{Appendix outline}

The appendix is organized as follows:
\begin{enumerate}[label=\Alph*.]
	\item Precise description of the \TaS strategy
	\item Proof of the time-uniform confidence regions guarantees for $\vmu$ (Lemma~\ref{lem:CR_mu})
	\item Proofs of the results on the sample complexity for Gaussian arms (Section~\ref{sec:sample_complexity})
	\item Proof of the non-asymptotic result (Theorem~\ref{thm:bound})
	\item Technical results associated to the proof of Theorem~\ref{thm:bound} (complements to Appendix~\ref{appendixD})
	\item Asymptotic analysis of \EBS (Theorems~\ref{thm:asymptotic_optimality_E} and~\ref{thm:asymptotic_optimality_as})
	\item Additional experiments to see the dependency of \EBS in $\delta$
\end{enumerate}

Without loss of generality (see~\citealp{GHMS19}), we assume that for any $a \in \Segent{K}$, $(X_{a, n})_{n \geq 1}$ is a sequence of random variables independent and identically distributed with distribution $\cN(\mu_a, 1)$, we set $\hat{\mu}_{a, n} = \frac{1}{n} \sum_{p\in [n]} X_{a, p}$ for all $n \geq 1$ and assume that 
\begin{equation} \label{eq:local_estimate}
	\forall t \geq K, \quad \hat{\mu}_{a} (t) = \hat{\mu}_{a, N_a(t)} \;.
\end{equation}


\section{The \TaS strategy}	\label{appendixA}

We recall the description of the \TaS strategy in Algorithm~\ref{algo:Track-and-Stop}. We use the notations of Section \ref{sec:strategy} and algorithm {\sc Optimal Weights} (Algorithm~\ref{algo:Newton_r} of Appendix \ref{appendixC3}) which efficiently computes the solution of optimization problem~\eqref{eq:def_T}.

\begin{center}
\begin{minipage}{13cm}
\begin{algorithm}[H]
	\caption{\TaS \label{algo:Track-and-Stop}}

	\DontPrintSemicolon
	\SetKwInOut{Input}{input}\SetKwInOut{Output}{output}
	\KwIn{confidence level $\delta$}
	\MyKwIn{threshold function $\beta(t, \delta)$}
	\KwOut{stopping time $\tau_\delta$}
	\MyKwOut{estimated best arm $\hat{a}_{\tau_\delta}$}
	
	\BlankLine
	Observe each arm once ; $t \gets K$ \;
	\For{$s = 0$ \KwTo $K-1$} {
		$\tilde{\vw}(s) \gets (1/K, \dots, 1/K)$
	}
	\While{$Z(t) \leq \beta(t, \delta)$}{
		\If{$U_t = \{a \in [K] \::\: N_a(t) < \sqrt{t} - K/2\} \neq \varnothing$}{
		Choose $A_{t+1} \in \argmin_{a \in U_t} N_a(t)$	\tcc*{forced exploration}
		}
		\Else{
		$\tilde{\vw}(t) \gets $ {\sc Optimal Weights}($\hat{\vmu}(t)$)\;
		Choose $A_{t+1} \in \argmin_{a \in \Segent{K}} N_a(t) - \sum_{s \in [t]} \tilde{w}_a(s)$ \tcc*{C-tracking}
		}
		Observe $Y_{A_{t+1}}$ and increase $t$ by $1$
	}
	$\tau_\delta \gets t$ ; $\hat{a}_{\tau_\delta} \gets \argmax_{a \in \Segent{K}} \hat{\mu}_a (t)$
\end{algorithm}
\end{minipage}
\end{center}

The presented algorithm uses C-tracking (the cumulative sums of the weights are tracked), but one can consider D-tracking for a direct track of the current weight (by replacing $\sum_{s \in [t]} \tilde{w}_a(s)$ by $t \tilde{w}_a(t)$).

%
%
%
%
%
%
%
%
%
%


\section{Proof of Lemma~\ref{lem:CR_mu}}	\label{appendixB}

By union bound we only have to show that for any $\gamma \in (0, 1)$ and $a \in \Segent{K}$:
\[ \P_{\vmu} \Big(\exists t \geq K : \abs{\hat{\mu}_a(t) - \mu_a} \geq C_\gamma(N_a(t)) \Big) \leq \gamma \;. \]

\noindent Fix $\gamma \in (0, 1)$ and $a \in \Segent{K}$. Note that as all arms are observed once at the beginning (see Algorithm~\ref{algo:ExplorationBiased}), we have $N_a(K) = 1$. Thus using Equation~\eqref{eq:local_estimate}:
\begin{align*}
	\P_{\vmu} \Big(\exists t \geq K : \abs{\hat{\mu}_a(t) - \mu_a}  \geq C_\gamma(N_a(t)) \Big) & = \P_{\vmu} \Big(\exists t \geq K : \abs{\hat{\mu}_{a, N_a(t)} - \mu_a} \geq C_\gamma(N_a(t)) \Big) \\
	                                                                                           & = \P_{\vmu} \Big(\exists n \in \N^* : \abs{\hat{\mu}_{a, n} - \mu_a} \geq C_\gamma(n) \Big) \;.
\end{align*}
Then we use a peeling trick (see for instance~\citealp{BLM16}):
\begin{align*}
	\P_{\vmu} \Big(\exists n \in \N^* : \abs{\hat{\mu}_{a, n} - \mu_a} \geq C_\gamma(n) \Big) & \leq \sum_{k \geq 0} \P \Big(\exists n \in [2^k, 2^{k+1}] : \Big|\frac{1}{p} \sum_{p\in [n]} (X_{a, p} - \mu_a) \Big| \geq C_\gamma(n) \Big)           \\
	                                                                                         & = \sum_{k \geq 0} \P \Big(\exists n \in [2^k, 2^{k+1}] : \Big| \sum_{p\in [n]} X_{a, p} - \mu_a \Big| \geq nC_\gamma(n) \Big)                          \\
	                                                                                         & \overset{\text{(a)}}{\leq} \sum_{k \geq 0} \P \Big(\exists n \in [0, 2^{k+1}] : \Big| \sum_{p\in [n]} X_{a, p} - \mu_a \Big| \geq 2^k C_\gamma(2^k) \Big) \\
	                                                                                         & \overset{\text{(b)}}{\leq} 2 \sum_{k \geq 0} \exp \Big( - \frac{(2^k C_\gamma(2^k))^2}{2 \times 2^{k+1}} \Big)                                        \\
	                                                                                         & = 2 \sum_{k \geq 0} \exp \big( -\log(2^{k+2}/\gamma) \big) \\
	                                                                                         &= 2\gamma \sum_{k \geq 0} \frac{1}{2^{k+2}} \\
	                                                                                         &= \gamma \;.
\end{align*}
(a) is obtained using the fact that $n \mapsto nC_\gamma(n)$ is non-decreasing and (b) is a well-known inequality for the sum of sub-Gaussian variables, see for instance \citet[Theorem 9.2]{LS20}.


\section{Proofs of results presented in Section~\ref{sec:sample_complexity}}	\label{appendixC}

In this appendix, we first prove Proposition~\ref{prop:links_r}, then we focus on the consequences developed in Section~\ref{sec:sample_complexity}.

\vskip 1em

\noindent \textbf{For the sake of simplicity, we assume that $a^* = 1$, except in the last section where there is no uniqueness assumption on the best arm of the bandits. }

\subsection{Solving the Optimization Problem}	\label{appendixC1}

\begin{proof}[Proof of Equation~\eqref{eq:def_g}]
	Let $\vv \in \Sigma_K$. One has:
	\begin{align*}
		g(\vmu, \vv) & = \inf_{\vlambda \in \Alt(\vmu)} \sum_{a \in \Segent{K}} v_a \frac{(\mu_a-\lambda_a)^2}{2}  = \frac{1}{2} \min_{a \neq 1} \inf_{\lambda_1 < \lambda_a}  v_1 (\mu_1-\lambda_1)^2 + v_a (\mu_a-\lambda_a)^2     \\
		             & = \frac{1}{2} \min_{a \neq 1} \inf_{\mu_1 \leq \lambda \leq \mu_a}  v_1 (\mu_1-\lambda)^2 + v_a (\mu_a-\lambda)^2  = \frac{1}{2} \min_{a \neq 1} \frac{v_1v_a}{v_1+v_a} (\mu_1-\mu_a)^2
	\end{align*}
	since the minimum is reached at $\lambda = \frac{v_1\mu_1 + v_a\mu_a}{v_1+v_a}$.
	\end{proof}

\begin{proof}[Proof of Proposition~\ref{prop:links_r}]  
	Let us define, for some $v_1 \in [0, 1]$:
	\begin{equation} C(v_1) = \max_{v_{2:K} \;: \;\vv \in \Sigma_K} \min_{a \neq 1} \frac{v_1v_a}{v_1+v_a} \Delta_a^2 \label{eq:defCofv}\end{equation}
	so that
	\begin{equation}	\label{eq:T=C(w1)}
		{T}^{-1} = \max_{\vv \in \Sigma_K} g(\vmu, \vv) = \frac{1}{2} \max_{v_1 \in [0, 1]} C(v_1) \;.
	\end{equation}

	Fix $v_1 \in [0, 1]$. The maximum in Equation~\eqref{eq:defCofv} is reached for $v_{2:K}$ such that  all the ($\frac{v_1v_a}{v_1+v_a} \Delta_a^2)_{a \neq 1}$ are equal, which happens when the $(v_a)_{a \neq 1}$ equalize those costs: $C$ is such that
	\[\forall a \neq 1, \quad C = \frac{v_1v_a}{v_1+v_a} \Delta_a^2 \]
	and hence:
	\begin{equation} 	\label{eq:link_wa_w1}
		\forall a \neq 1, \quad v_a = \frac{v_1C}{v_1 \Delta_a^2 - C} \;. 
	\end{equation}
	The fact that $\vv \in \Sigma_K$ yields:
	\begin{equation}	\label{eq:rel_w1_C}
		\Phi(v_1, C) := v_1 +\sum_{a \neq 1} \frac{v_1 C}{v_1 \Delta_a^2 - C} -1 = 0\;.
	\end{equation}
	By the implicit function theorem, there exists a mapping $C(v_1)$ such that $\Phi(v_1, C(v_1))=0$ and
	\begin{align*}
		C'(v_1) & = -\frac{\frac{\partial \Phi}{\partial v_1}\big(v_1, C(v_1)\big)}{\frac{\partial \Phi}{\partial C}\big(v_1, C(v_1)\big)} = -\frac{1 + \sum_{a \neq 1} \frac{C(v_1) (v_1\Delta_a^2-C(v_1))-v_1C(v_1)\Delta_a^2}{(v_1\Delta_a^2-C(v_1))^2}}
		{v_1^2\sum_{a \neq 1} \frac { \Delta_a^2} {(v_1\Delta_a^2-C(v_1))^2} }
		= -\frac{ 1 - \sum_{a \neq 1} \frac{1}{(v_1\Delta_a^2/C(v_1) -1)^2} }{v_1^2\sum_{a \neq 1} \frac { \Delta_a^2} {(v_1\Delta_a^2-C(v_1))^2} }\;.
	\end{align*}
	Hence $C(v_1)$ is a smooth non-negative function with a continuous derivative. By Equation~\eqref{eq:T=C(w1)}, it vanishes when $v_1\to 0$ and $v_1\to 1$, and hence its maximum is reached at a point $w_1$ where $C'(w_1)=0$.
	Define $r = w_1/C(w_1)$ by the relation
	\[ C'(w_1)=0 \quad \iff \quad 1 - \sum_{a \neq 1} \frac{1}{\big(\frac{w_1}{C(w_1)}\Delta_a^2 -1\big)^2} = 0 \]
	$r$ is the unique solution of $\phi_{\vmu}(r) = 0$.

	\noindent Equations~\eqref{eq:link_r_w1},~\eqref{eq:def_wa} and~\eqref{eq:link_r_T} can be respectively derived from~\eqref{eq:rel_w1_C},~\eqref{eq:link_wa_w1} and~\eqref{eq:T=C(w1)}.
	It remains to obtain Equation~\eqref{eq:w1=norm2_wa} by combining Equation~\eqref{eq:def_wa} and the characterization $\phi_{\vmu}(r) = 0$:
	\[ \sum_{a \neq 1} w_a^2 = w_1^2 \sum_{a \neq 1} \frac{1}{(r\Delta_a^2-1)^2} = w_1^2 (\phi_{\vmu}(r) + 1) = w_1^2 \;. \]
	\end{proof}

\begin{proof}[Proof of Corollary~\ref{cor:weights_order}]
	When $a, b$ are suboptimal, the result is a direct consequence of Equation~\eqref{eq:def_wa} of Proposition~\ref{prop:links_r}. 
	It remains to see that $w_1 > \max_{a \neq 1} w_a$, which is a direct consequence of Equation~\eqref{eq:w1=norm2_wa} and the fact that all weights are positive. 
\end{proof}

\subsection{Proof of Proposition~\ref{prop:bounds}} \label{sec:proofProp4}

Defining $q_a = \frac{1}{r\Delta_a^2-1}$ for $a \neq 1$, we will use that, as $\phi_{\vmu}(r) = 0$, the $(q_a^2)_{a \neq 1}$ are positive and sum to $1$, hence for any $a \neq 1$ one has $q_a \leq 1$ (with strict inequality when $K \geq 3$). 

Let us begin with Equation~\eqref{eq:bounds_w1}. As we assume $a^* = 1$, $w_{\max} = w_1$ by Corollary~\ref{cor:weights_order}. Using Equation~\eqref{eq:link_r_w1} of Proposition~\ref{prop:links_r} one has:
\begin{itemize}
	\item on the one hand
	      \begin{align*}
		      w_1 & = \Big(1+\sum_{a \neq 1} \frac{1}{r\Delta_a^2 - 1}\Big)^{-1}        &  & \text{by Equation~\eqref{eq:link_r_w1} of Proposition~\ref{prop:links_r}} \\
		            & \leq \Big(1+\sum_{a \neq 1} \frac{1}{(r\Delta_a^2 - 1)^2}\Big)^{-1} &  & \text{as $q_a \leq 1$}                                                    \\
		            & = \frac{1}{2}                                                       &  & \text{as } \phi_{\vmu}(r) = 0
	      \end{align*}
	      giving the upper bound ;

	\item on the other hand, by the Cauchy-Schwarz inequality:
	      \begin{equation*}
		      w_1 \geq \Bigg(1+\sqrt{ (K-1) \sum_{a \neq 1} \frac{1}{(r\Delta_a^2 - 1)^2} } \Bigg)^{-1} = \frac{1}{1+\sqrt{K-1}}\;. \label{eq:LB_w1}
	      \end{equation*}
\end{itemize}

\noindent We now prove Inequalities~\eqref{eq:bounds_r} : 
\begin{itemize}
	\item since $q_a \leq 1$ or equivalently $r\Delta_a^2 \geq 2$ for every $a \neq 1$,
	      \begin{equation*}	\label{eq:LB1_r}
		      r \geq \frac{2}{\Delta_{\min}^2} \;.
	      \end{equation*}

	\item since $\overline{\Delta^2} = \frac{1}{K-1} \sum_{a \neq 1} \Delta_a^2$, by convexity of $x \mapsto \frac{1}{(rx-1)^2}$:
	      \[ \frac{1}{K-1}\sum_{a \neq 1} \frac{1}{ \big(\frac{1+\sqrt{K-1}}{\overline{\Delta^2}}\Delta_a^2 - 1\big)^2 }  \geq  \frac{1}{\big(\frac{1+\sqrt{K-1}}{\overline{\Delta^2}} \overline{\Delta^2} - 1\big)^2}  =  \frac{1}{K-1} \]
	      and hence $\phi_{\vmu}(\frac{1+\sqrt{K-1}}{\overline{\Delta^2}}) \geq 0$, which by decreasing of $\phi_{\vmu}$ (Lemma~\ref{lem:convexity_phi}) gives $ r\geq \frac{1+\sqrt{K-1}}{\overline{\Delta^2}}$.

	\item one can also check that
	      \[ \phi_{\vmu} \Big( \frac{1+\sqrt{K-1}}{\Delta_{\min}^2} \Big) = \sum_{a \neq 1} \frac{1}{\big( \frac{1+\sqrt{K-1}}{\Delta_{\min}^2} \Delta_a^2-1 \big)^2} - 1 \leq 0 \]
	      so that $r \leq \frac{1+\sqrt{K-1}}{\Delta_{\min}^2}$.
\end{itemize}

\noindent Finally, combining the obtained inequalities with Equation~\eqref{eq:link_r_T} yields Equation~\eqref{eq:bounds_T}.

\smallskip

To conclude this section, we discuss about the tightness of the proven inequalities.
\begin{itemize}[topsep=0cm]
	\item First note that when $K=2$, lower and upper bounds match in Inequalities~\eqref{eq:bounds_r},~\eqref{eq:bounds_w1} and~\eqref{eq:bounds_T}. In that case the problem is easy as we always have $\vw = (0.5, 0.5)$.

	\item In fact, equalities $r = 2/\Delta_{\min}^2$, $w_1 = 1/2$ and $T = 8/\Delta_{\min}^2$ occur if and only if $K=2$. This is because the $(q_a)_{a \neq 1}$ are positive and sum to $1$ (thus $q_2=1$ only when $K=2$). The presence of other arms thus increases $r$ and $T$ while decreases $w_1$.

	\item If there is at least $3$ arms, then the remaining equalities $w_1 = (1+\sqrt{K-1})^{-1}$, $r = (1+\sqrt{K-1})/\ov{\Delta^2}$, $r = (1+\sqrt{K-1})/\Delta_{\min}^2$ and $T = 2\big(1+\sqrt{K-1}\big)^2 / \Delta_{\min}^2$ are reached if and only if $\Delta_{\min} = \Delta_{\max}$, or in other words $\Delta_2 = \dots = \Delta_K$. Indeed, the condition can be obtained by studying the equality cases in the proof above, using the equality case of the Cauchy-Schwarz inequality for $w_1$, the strict convexity of $x \mapsto \frac{1}{(rx-1)^2}$ and the decreasing of $\phi_{\vmu}$ for $r$ and finally the link $T = 2r/w_1$ for $T$. Note that in that case, $T$ grows linearly with $K$.
\end{itemize}

\subsection{Computing $r$}	\label{appendixC3}

At the sight of Proposition~\ref{prop:links_r}, it suffices to compute $r$ to obtain the values of both the optimal weight vector and the sample complexity.

The function $\phi_{\vmu}$ is convex and strictly decreasing on $(1/\Delta_{\min}^2, +\infty)$ (Lemma~\ref{lem:convexity_phi}). Hence, when initialized with a value $r_0<r$, the iterates of a Newton procedure remain smaller than $r$.
The lower bound of Inequalities~\eqref{eq:bounds_r} of Proposition~\ref{prop:bounds} permits such an initialization.
The convergence is quadratic (the number of correct digits roughly doubles at every step), which implies that a few iterations are sufficient to guarantee machine precision. The cost of the algorithm can hence be considered proportional to that of evaluating $\phi_{\vmu}(r)$, which is linear in the number of arms. See Algorithm~\ref{algo:Newton_r} for details.

\begin{center}
\begin{minipage}{11cm}
\begin{algorithm}[H]
	\caption{\textsc{Optimal Weights}} 	\label{algo:Newton_r}

	\DontPrintSemicolon
	\SetKwInOut{Input}{input}\SetKwInOut{Output}{output}
	\KwIn{bandit $\vmu \in \cG^*$ with best arm $1$}
	\MyKwIn{tolerance parameter $\mathrm{tol}$ (typically $10^{-10}$)}
	\KwOut{optimal weight vector $\vw$}
	\MyKwOut{characteristic time $T$}

	\BlankLine
	\For{$a = 2$ \KwTo $K$}{
		$\Delta_a \gets \mu_1-\mu_a$
	}
	$\displaystyle{\phi_{\vmu}(r) \gets \sum_{a \neq 1} \frac{1}{(r\Delta_a^2 -1)^2} - 1}$ ; $\displaystyle{\phi_{\vmu}'(r) \gets -2 \sum_{a \neq 1} \frac{\Delta_a^2}{(r\Delta_a^2 -1)^3}}$ \;
	$\displaystyle{r \gets \max\Big( \frac{2}{\Delta_{\min}^2}, \frac{1+\sqrt{K-1}}{\overline{\Delta^2}} \Big)}$ \;
	\While{$|\phi_{\vmu}(r)| \geq \mathrm{tol}$}{
		$\displaystyle{r \gets r -  \frac{\phi_{\vmu}(r)}{\phi_{\vmu}'(r)}}$
	}
	$\displaystyle w_1 \gets \biggl(1+\sum_{a \neq 1} \frac{1}{r\Delta_a^2 - 1}\biggr)^{-1}$ \;
	\For{$a = 2$ \KwTo $K$}{
		$\displaystyle w_a \gets \frac{w_1}{r\Delta_a^2 - 1}$
	}
	$T \gets 2 \frac{r}{w_1}$
\end{algorithm}
\end{minipage}
\end{center}

\subsection{On the monotonicity of the $\min$-$\max$ problem}	\label{appendixC4}



In this section we prove Lemmas~\ref{lem:moving_one_suboptimal_arm},~\ref{lem:moving_optimal_arm} and~\ref{lem:moving_worst_arms}, and then use those Lemmas to prove Proposition~\ref{prop:optimistic_weights}. We recall that we assume $K \geq 3$ in this section (note that Proposition~\ref{prop:optimistic_weights} is trivial when $K = 2$). 

\begin{proof}[Proof of Lemma~\ref{lem:moving_one_suboptimal_arm}] 
	\phantom{}
	\begin{enumerate}
		\item Since
		      \[ \sum_{a \neq 1} \frac{1}{(r{\Delta'_a}^2-1)^2} < \sum_{a \neq 1} \frac{1}{(r\Delta_a^2-1)^2} =1\;,\]
		      it holds that $r'<r$. It implies that for $a\notin \acco{1, b}$ one has:
		      \[ \frac{1}{r'{\Delta'_a}^2-1} > \frac{1}{r\Delta_a^2-1} \;. \]
		      As $K \geq 3$, such an arm $a$ exists and hence as $\phi_{\vmu}(r) = 0 = \phi_{\vmu'}(r')$:
		      \[ \frac{1}{r'{\Delta'_b}^2-1} < \frac{1}{r\Delta_b^2-1} \]
		      or equivalently $r'{\Delta'_b}^2-1 > r\Delta_b^2-1$.

		      Combining those inequalities with Equation~\eqref{eq:def_wa} of Proposition~\ref{prop:links_r}, we have for all $a \notin \{1, b\}$:
		      \[ \frac{w'_a}{w'_b} = \frac{r'{\Delta'_b}^2-1}{r'{\Delta'_a}^2-1} > \frac{r\Delta_b^2-1}{r\Delta_a^2-1} = \frac{w_a}{w_b} \;. \]
		      Besides, $w'_1/w'_b = r'{\Delta'_b}^2-1> r\Delta_b^2-1 = w_1/w_b$. Hence,
		      \[ \frac{1-w'_b}{w'_b} = \sum_{a\neq b}\frac{w'_a}{w'_b} > \sum_{a\neq b}\frac{w_a}{w_b}= \frac{1-w_b}{w_b} \]
		      and thus $w'_b<w_b$.

		\item For any $\vnu \in \cG^*$ with best arm $1$, one can see $\vw(\vnu)$ or its components as a function of $\vDelta^2(\vnu)$. Fix $a \notin \{ 1, b \}$ and define $F_a(\vDelta^2(\vnu))$ as
		      \[ F_a(\vDelta^2(\vnu)) = \frac{1}{w_a(\vnu)} = \frac{r(\vnu)\Delta_a(\vnu)^2-1}{w_1(\vnu)} = (r(\vnu)\Delta_a^2-1) + \sum_{c \neq 1} \frac{r(\vnu)\Delta_a^2-1}{r(\vnu)\Delta_c^2-1} \]
		      where the right-inequalities are derived from Equations~\eqref{eq:link_r_w1} and~\eqref{eq:def_wa} of Proposition~\ref{prop:links_r}. Recall that $r(\vnu)$ also depends uniquely on the gaps, as the unique solution of $\phi_{\vnu} = 0$. In the following calculations we write $r$ for $r(\vnu)$ but the dependency with respect to the gaps is crucial.

		      Fix $d_1 = 0$ and $d_a = \Delta_a^2$ for $c \neq \{ 1, b \}$. We want to see the change of $F_a$ with respect to $d_b = \Delta_b^2$. We can take the partial derivative:
		      \begin{align*}
			      \frac{\partial F_a}{\partial d_b} & = \frac{\partial r}{\partial d_b} d_a + \sum_{c \neq 1} \left[\frac{\frac{\partial r}{\partial d_b}d_a}{rd_c-1} - \frac{rd_a-1}{(rd_c-1)^2}\left(\frac{\partial r}{\partial d_b}d_c\right) \right] - \frac{rd_a-1}{(rd_b-1)^2} r \\
			                                        & = \frac{\partial r}{\partial d_b}d_a \Big(1 + \sum_{c \neq 1}  \frac{1}{rd_c-1} - \frac{rd_c}{(rd_c-1)^2} \Big) + \frac{\partial r}{\partial d_b} \sum_{c \neq 1} \frac{d_c}{(rd_c-1)^2} - \frac{rd_a-1}{(rd_b-1)^2} r           \\
			                                        & = \frac{\partial r}{\partial d_b}d_a\sum_{c \neq 1} \underbrace{\frac{1 + (rd_c-1) - rd_c}{(rd_c-1)^2}}_{=0} + \frac{\partial r}{\partial d_b}\sum_{c \neq 1} \frac{d_c}{(rd_c-1)^2} - \frac{rd_a-1}{(rd_b-1)^2}r                \\
			                                        & = \frac{\partial r}{\partial d_b}\sum_{c \neq 1} \frac{d_c}{(rd_c-1)^2} - \frac{rd_a-1}{(rd_b-1)^2}r
		      \end{align*}
		      (to obtain the third equality, we used that $\sum_{c \neq 1} \frac{1}{(rd_c-1)^2} = 1$ by definition of $r$).

		      It remains to see that $\frac{\partial r}{\partial d_b}$ is nonpositive, that is that $r$ is nondecreasing when $\Delta_b$ increases. In fact, we already noticed that by showing that $r' < r$ in the first part of the proof of Lemma~\ref{lem:moving_one_suboptimal_arm}. Note that one can also use the implicit function theorem to obtain
		      \begin{equation*}	\label{eq:dr}
			      \frac{\partial r}{\partial d_b} = -\frac{r(rd_b-1)^{-3}}{\sum_{c \neq 1} d_c(rd_c-1)^{-3}} < 0 \;.
		      \end{equation*}
		      Hence $\frac{\partial F_a}{\partial d_b} < 0$, so that as ${\Delta_b}' > \Delta_b$:
		      \[ \frac{1}{w_a} = F_a(\vDelta^2) > F_a({\vDelta'}^2) = \frac{1}{w_a'} \quad \text{giving} \quad w_a' > w_a \;. \]

		\item Using Equations~\eqref{eq:link_T_g} and~\eqref{eq:def_g}:
		      \[ {T'}^{-1} = \frac{1}{2} \min_{a \neq 1} \frac{w'_1 w'_a}{w'_1 + w'_a} {\Delta'_a}^{2} \geq \frac{1}{2} \min_{a \neq 1} \frac{w'_1 w'_a}{w'_1 + w'_a} \Delta_a^{2} > \frac{1}{2} \min_{a \neq 1} \frac{w_1 w_a}{w_1 + w_a} \Delta_a^{2} = T^{-1} \;, \]
		      the first inequality comes from the assumption on $\vmu$ and $\vmu'$, and the second is a consequence of the uniqueness of the optimal weight vector $\vw$ and the fact that $\vw \neq \vw'$, as previously obtained.

	\end{enumerate}
	\end{proof}

Before proving Lemmas~\ref{lem:moving_optimal_arm} and~\ref{lem:moving_worst_arms}, we show the following result.

\begin{lemma}\label{lem:scaling} Assume that there exists $\kappa > 0$ such that $\Delta'_a = \kappa\Delta_a$ for any $a \neq 1$. Then ${\vw}' = \vw$.
\end{lemma}

\begin{proof}[Proof of Lemma~\ref{lem:scaling}]
	As $r$ is the unique solution of $\phi_{\vmu}(r) = 0$, one has:
	\[ 0 = \phi_{\vmu}(r) = \sum_{a \neq 1} \frac{1}{(r\Delta_a^2-1)^2} - 1 = \sum_{a \neq 1} \frac{1}{(\frac{r}{\kappa^2}(\kappa\Delta_a)^2-1)^2} - 1 = \sum_{a \neq 1} \frac{1}{(\frac{r}{\kappa^2}{\Delta'_a}^2-1)^2} - 1 = \phi_{\vmu'}\big(\frac{r}{\kappa^2}\big) \]
	and thus $r' = r/\kappa^2$.

	\noindent This implies $r\Delta_a^2 = {r'\Delta'_a}^2$ for any $a \neq 1$, hence $\vw' = \vw$ by Equations~\eqref{eq:link_r_w1} and~\eqref{eq:def_wa} of Proposition~\ref{prop:links_r}.
	\end{proof}

\begin{proof}[Proof of Lemma~\ref{lem:moving_optimal_arm}]
	Let us rescale the gaps of $\vmu'$ to obtain the same maximal gap, by multiplying by constant $\kappa = \frac{\Delta_{\max}}{\Delta_{\max} + d}$. Denoting by $\vmu''$ the obtained bandit, with $\vDelta'' = \vDelta(\vmu'') = \kappa \vDelta'$ and $\vw'' = \vw(\vmu'')$, we have $\vw'' = \vw'$ by Lemma~\ref{lem:scaling}. Let $a$ be (one of) the worst arm of $\vmu$, such that $\Delta_a = \Delta_{\max}$. Then
	\[ \Delta''_{\max} = \Delta''_a = \kappa \Delta'_a = \frac{\Delta_a}{\Delta_a + d}(\Delta_a + d) = \Delta_a = \Delta_{\max} \]
	and for any $b \neq 1$, one has $\Delta_b \leq \Delta_a$ so that the nondecreasing of $x \mapsto \frac{x}{x+d}$ leads to:
	\[ \Delta''_b = \kappa \Delta'_b = \frac{\Delta_a}{\Delta_a + d}(\Delta_b + d) \geq \frac{\Delta_b}{\Delta_b + d}(\Delta_b + d) = \Delta_b \;. \]

	\noindent Now we can apply Lemma~\ref{lem:moving_one_suboptimal_arm} to every arm $b \notin \{ 1,\, a\}$ to go from $\vmu$ to $\vmu''$, and by Point~\ref{lem:moving_one_suboptimal_arm2} we know that those transformations can only increase $w_a$, so that by Corollary~\ref{cor:weights_order}
	\[ w'_{\min} = w'_a = w''_a \geq w_a = w_{\min} \;. \]

	\noindent If in addition there exists an arm $b$ for which $\Delta_b < \Delta_a$, then strict inequality $\Delta_b < \Delta''_b$ occurs in the above inequality and hence Lemma~\ref{lem:moving_one_suboptimal_arm} gives a strict increasing of $w_{\min}$.
	\end{proof}

\begin{proof}[Proof of Lemma~\ref{lem:moving_worst_arms}] 
	Using scaling argument from Lemma~\ref{lem:scaling}, like in the proof of Lemma~\ref{lem:moving_optimal_arm}, we can scale $\vmu'$ to keep gap between arm $1$ and arms of $B$ unchanged. That would increase the gaps of all the other arms which in consequence, using Point~\ref{lem:moving_one_suboptimal_arm2} of Lemma~\ref{lem:moving_one_suboptimal_arm}, would mean that corresponding $w_{\min}$ increases.
	\end{proof}

Finally we can prove that Algorithm~\ref{algo:OptimisticWeights} correctly computes the optimistic bandit.
\begin{proof}[Proof of Proposition~\ref{prop:optimistic_weights}] We stick to the notation of Algorithm~\ref{algo:OptimisticWeights}, and first observe that $\vw = \vw(\tilde{\vmu})$. 
	When $\text{minUB} \geq \text{maxLB}$ the algorithm returns a constant bandit and $\vw = (1/K, \dots, 1/K)$ which is its optimal weight vector by convention. As all weight vectors belong to $\Sigma_K$, the result is clear.

	Now assume that $\text{minUB} < \text{maxLB}$ and fix $\vnu \in \cC \cR$. If $\vnu$ as several optimal arms, then $w_{\min}(\vnu) = 0$ so that trivially $w_{\min}(\vnu) \leq w_{\min}(\tilde{\vmu})$. Assume now that $\vnu$ has a unique optimal arm denoted by $a$. Note that $a \in \text{PotentialBest}$, so that we will show that $w_{\min}(\vnu) \leq w_{\min}(\tilde{\vmu}^{\text{test}(a)})$ by transforming $\vnu$ to $\tilde{\vmu}^{\text{test}(a)}$ with changes that will only increase the quantity of interest $w_{\min}$. Remark that the value of $w_{\min}$ is the vector value associated to any of the worst arms of a bandit due to Corollary~\ref{cor:weights_order}.
	The procedure, illustrated in Figure~\ref{fig:illustration_procedure_proof_prop1}, is the following:
	\begin{enumerate}
		\item Transform $\vnu$ into $\vnu^{(1)}$ by increasing arm $a$ so that $\vnu^{(1)}_a = \ov{\mu}_a$. Using Lemma~\ref{lem:moving_optimal_arm}, one has $w_{\min}(\vnu^{(1)}) \geq w_{\min}(\vnu)$.

		\item Transform $\vnu^{(1)}$ into $\vnu^{(2)}$ by decreasing, for each arm $b \neq a$, $\mu_b$ to $\max(\underline{\mu}_b, \vnu_{\min})$. By several applications of Lemma~\ref{lem:moving_one_suboptimal_arm}, one has $w_{\min}(\vnu^{(2)}) \geq w_{\min}(\vnu^{(1)})$ (remark that imposing to stay above $\vnu_{\min}$ ensures that the associated worst arm stays one of the worst arms at each modification).

		\item Transform $\vnu^{(2)}$ into $\vnu^{(3)}$ by increasing all the worst arms to $\text{minUB}$. By Lemma~\ref{lem:moving_worst_arms}, one has $w_{\min}(\vnu^{(3)}) \geq w_{\min}(\vnu^{(2)})$.
	\end{enumerate}
	We now have $\vnu^{(3)} = \tilde{\vmu}^{\text{test}(a)}$ so that $w_{\min}(\vnu) \leq w_{\min}(\tilde{\vmu}^{\text{test}(a)})$.
	We thus showed that
	\[ \max_{\vnu \in \cC \cR} w_{\min}(\vnu) = \max_{a \in \text{PotentialBest}} w_{\min}(\tilde{\vmu}^{\text{test}(a)}) = w_{\min}(\tilde{\vmu}) \;,\]
	where the last inequality comes from the procedure defining $\tilde{\vmu}$.
	\end{proof}

\begin{figure}
	\centering
	\includegraphics[width=0.24\linewidth]{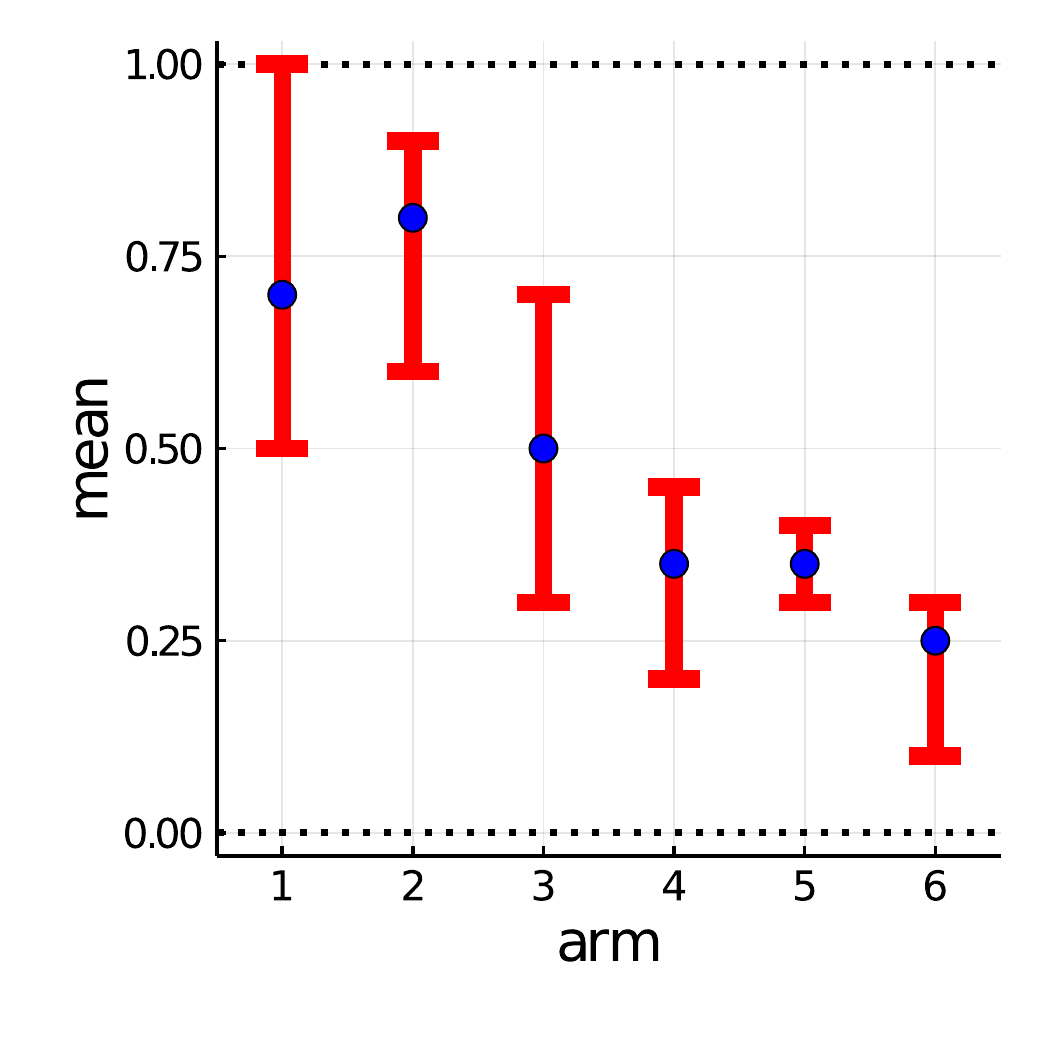}
	\includegraphics[width=0.24\linewidth]{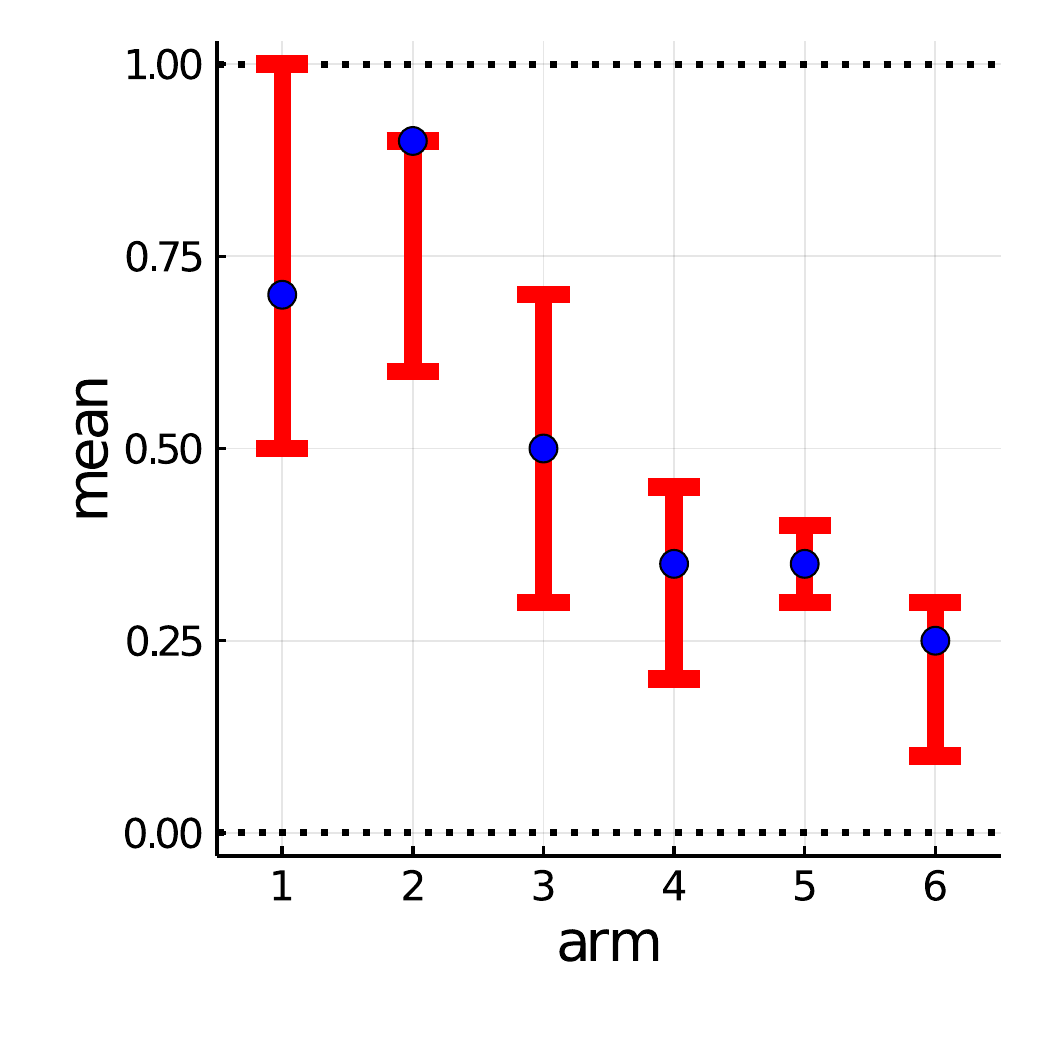}
	\includegraphics[width=0.24\linewidth]{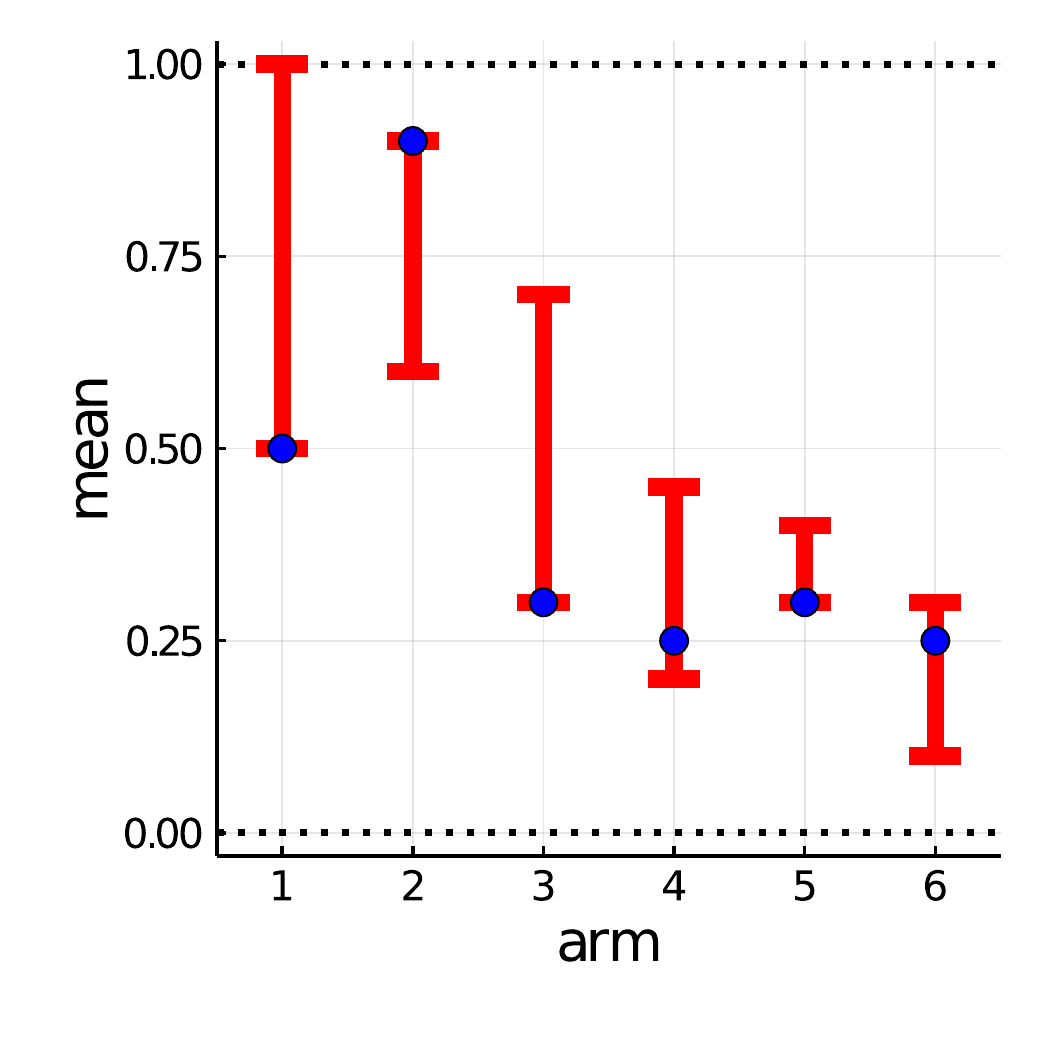}
	\includegraphics[width=0.24\linewidth]{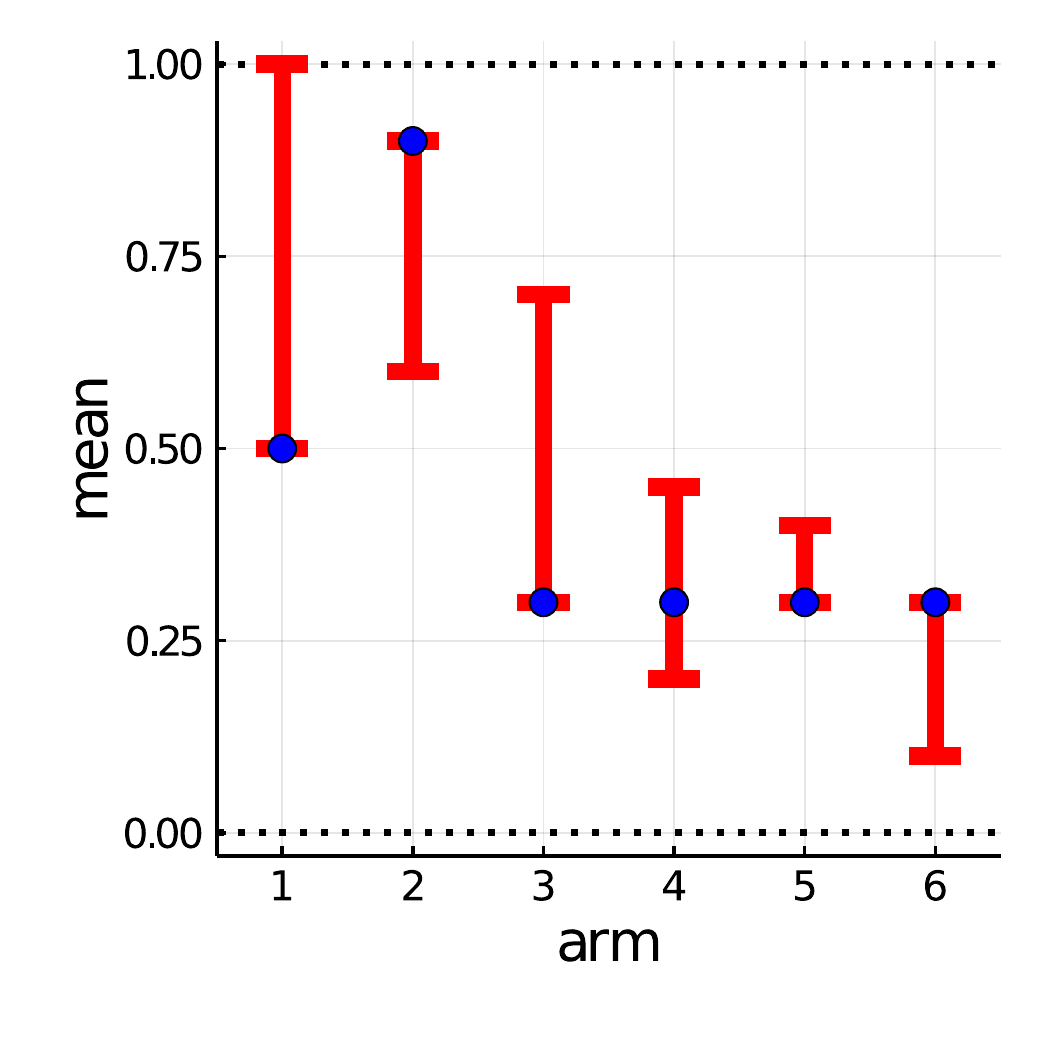}
	\caption{Transformations in the proof of Proposition~\ref{prop:optimistic_weights}, for some instance bandit $\vnu$. From left to right: $\vnu$, $\vnu^{(1)}$, $\vnu^{(2)}$, $\vnu^{(3)} = \tilde{\vmu}^{\text{test}(2)}$	\label{fig:illustration_procedure_proof_prop1}}
\end{figure}

\subsection{Proof of Theorem~\ref{thm:regularity}}

We have that
\[ \phi_{\vmu'}\Big( \frac{r}{1+\varepsilon}\Big) = \sum_{a \neq 1} \frac{1}{\big(\frac{r}{1+\varepsilon}{\Delta'_a}^2 -1\big)^2} - 1\geq \sum_{a \neq 1} \frac{1}{\big(\frac{r}{1+\varepsilon}\Delta_a^2(1+\varepsilon) -1\big)^2} - 1 = \phi_{\vmu}(r) = 0 \]
and
\[ \phi_{\vmu'}\Big( \frac{r}{1-\varepsilon}\Big) = \sum_{a \neq 1} \frac{1}{\big(\frac{r}{1-\varepsilon}{\Delta'_a}^2 -1\big)^2} - 1 \leq \sum_{a \neq 1} \frac{1}{\big(\frac{r}{1-\varepsilon}\Delta_a^2(1-\varepsilon) -1\big)^2} - 1 = \phi_{\vmu}(r) = 0 \]
hence by monotonicity of $\phi_{\vmu'}$ and definition of $r'$:
\[ \frac{r}{1+\varepsilon} \leq r' \leq  \frac{r}{1-\varepsilon} \;. \]
Consequently, for every $a \neq 1$, $r'{\Delta'_a}^2 \leq (1+\eta) r\Delta_a^2$ for $1+\eta = (1+\varepsilon)/(1-\varepsilon)$, and
\begin{align*}\frac{1}{r'{\Delta'_a}^2-1} \geq \frac{1}{\big(r\Delta_a^2-1\big)\big(1+\frac{\eta r\Delta_a^2}{r\Delta_a^2-1}\big)} & \geq \frac{1}{r\Delta_a^2-1}\Big(1- \frac{\eta r\Delta_a^2}{r\Delta_a^2-1}\Big)  = \frac{1}{r\Delta_a^2-1} - \eta\frac{1}{r\Delta_a^2-1} - \eta \frac{1}{(r\Delta_a^2-1)^2}\end{align*}
so that
\begin{align*}
	(w'_1)^{-1} & = 1+ \sum_{a \neq 1} \frac{1}{r'{\Delta'_a}^2-1}                                                                              \\
	            & \geq 1+ (1-\eta) \sum_{a \neq 1} \frac{1}{r\Delta_a^2-1} - \eta \underbrace{\sum_{a \neq 1} \frac{1}{(r\Delta_a^2-1)^2}}_{=1} \\
	            & = (1-\eta) w_1^{-1} = \frac{1-3\varepsilon}{1-\varepsilon}w_1^{-1} \geq (1-3\varepsilon) w_1^{-1} \;.
\end{align*}
Furthermore, $r\Delta_a^2 \geq 2$ (see the lower bound in Inequalities~\eqref{eq:bounds_r} of Proposition~\ref{prop:bounds}), hence $\frac{r \Delta_a^2}{r\Delta_a^2-1} \leq 2$ by decreasing of $x \mapsto \frac{x}{x-1}$ on $(2, +\infty)$. Thus, for every $\eta\leq 1/4$, $u = \eta \frac{r \Delta_a^2}{r\Delta_a^2-1} \leq 1/2$ and $\frac{1}{1-u} \leq 1+2u$. One has $r'{\Delta'_a}^2 \geq (1-\eta) r\Delta_a^2$ for $1-\eta = (1-\varepsilon)/(1+\varepsilon)$, and one checks that $\eta\leq 1/4$ for $\varepsilon\leq 1/7$, hence
\begin{align*}
	\frac{1}{r'{\Delta'_a}^2-1} & \leq \frac{1}{\big(r\Delta_a^2-1\big)\left(1-\frac{\eta r\Delta_a^2}{r\Delta_a^2-1}\right)} \leq \frac{1}{r\Delta_a^2-1} \Big(1 +2 \frac{\eta r\Delta_a^2}{r\Delta_a^2-1}\Big)  = \frac{1}{r\Delta_a^2-1} + 2\eta\frac{1}{r\Delta_a^2-1} +2 \eta \frac{1}{(r\Delta_a^2-1)^2}
\end{align*}
Consequently,
\begin{align*}
	(w'_1)^{-1} & = 1+ \sum_{a \neq 1} \frac{1}{r'{\Delta'_a}^2-1}                                                                                \\
	            & \leq 1+ (1+2\eta) \sum_{a \neq 1} \frac{1}{r\Delta_a^2-1} +2 \eta \underbrace{\sum_{a \neq 1} \frac{1}{(r\Delta_a^2-1)^2}}_{=1} \\
	            & = (1+2\eta) w_1^{-1}  = \frac{1+5\varepsilon}{1+\varepsilon} w_1^{-1} \leq (1+5\varepsilon) w_1^{-1}\;.
\end{align*}

To summarize, for $\varepsilon \leq 1/7$, by Equation~\eqref{eq:link_r_T} of Proposition~\ref{prop:links_r}, on the one hand:
\[ T' = 2r' {w_1'}^{-1} \geq 2 \times \frac{r}{1+\epsilon} \times \frac{1-3\epsilon}{1-\epsilon} {w_1}^{-1} = \frac{1-3\varepsilon}{1+\varepsilon^2} \times T \geq (1-3\varepsilon) T \]
and on the other hand
\[ T' = 2r' (w_1')^{-1} \leq 2 \times \frac{r}{1-\varepsilon} \times \frac{1+5\varepsilon}{1+\varepsilon} w_1^{-1} = \frac{1+5\varepsilon}{1-\varepsilon^2} \times T \leq (1+6\varepsilon) T \]
as $1+5\varepsilon\leq (1+6\varepsilon) (1-\varepsilon^2)$.

\noindent We also have
\[ (1-5\varepsilon)w_1\leq \frac{w_1}{1+5\varepsilon}\leq  w'_1 \leq \frac{w_1}{1-3\varepsilon} \leq (1+6\varepsilon) w_1 \]
which yields by Equation~\eqref{eq:def_wa} of Proposition~\ref{prop:links_r}, for any $a \neq 1$:
\begin{multline*}
	(1-10\varepsilon) w_a \leq \frac{w_1/(1+5\varepsilon)}{(r\Delta_a^2-1) (1+\frac{2\varepsilon}{1+\varepsilon})} \leq w'_a = \frac{w'_1}{r'{\Delta'_a}^2-1} \\
	\leq \frac{w_1/(1-3\varepsilon)}{(r\Delta_a^2-1) (1-2\frac{2\varepsilon}{1+\varepsilon})} = \frac{1+\varepsilon}{(1-3\varepsilon)^2}w_a \leq (1+10\varepsilon) w_a \;.
\end{multline*}

\subsection{Proof of Proposition~\ref{prop:control_g}}

We will prove Proposition~\ref{prop:control_g} by combining two Lemmas. Note that in this section $\vmu$ and $\vmu'$ are general bandits, with possibly more than one best arm.

\begin{lemma} \label{lem:moving_mu} Let $\vmu, \vmu' \in \cG$ and $\vv \in \Sigma_K$ be any optimal vector. Then:
	\[ g(\vmu', \vv) \geq g(\vmu, \vv) - \varepsilon/2 \]
	where $\varepsilon = \norm{\vmu-\vmu'}_\infty$.
\end{lemma}

\begin{proof}
	\phantom{}
	\begin{itemize}
		\item Assume first that $\vmu$ and $\vmu'$ have a common best arm. Without loss of generality we assume that this arm is $1$. Then:
		      \begin{align*}
			      g(\vmu', \vv) - g(\vmu, \vv) & = \frac{1}{2} \min_{a \neq 1} \frac{v_1v_a}{v_1+v_a} {\Delta'_a}^2 - \frac{1}{2} \min_{b \neq 1} \frac{v_1v_b}{v_1+v_b} \Delta_b^2 &  & \text{by Equation~\eqref{eq:def_g}} \\
			                                   & = \frac{1}{2} \min_{a \neq 1} \max_{b \neq 1} \frac{v_1v_a}{v_1+v_a} {\Delta'_a}^2 - \frac{v_1v_b}{v_1+v_b} \Delta_b^2                                                      \\
			                                   & \geq \frac{1}{2} \min_{a \neq 1} \frac{v_1v_a}{v_1+v_a} \big( {\Delta'_a}^2 - \Delta_a^2 \big)                                     &  & \text{taking } b = a \;.
		      \end{align*}

		      Then for any $a \neq 1$, one has:
		      \[ \abs{\Delta_a - \Delta'_a} = \abs{(\mu_1-\mu_1') - (\mu_a-\mu_a')} \leq \abs{\mu_1-\mu_1'} + \abs{\mu_a-\mu_a'} \leq 2\varepsilon \]
		      from which we obtain, using that the gaps are in $[0, 1]$ in $\cG$
		      \[ \abs{\Delta_a^2 - {\Delta'_a}^2} = \abs{\Delta_a - \Delta'_a} (\Delta_a + \Delta_a') \leq 4\varepsilon \;. \]
		      As $\vv$ is an optimal vector, we have $0 \leq v_a \leq v_1 \leq \frac{1}{2}$ using Equation~\eqref{eq:bounds_w1}, so that:
		      \[ \frac{v_1v_a}{v_1+v_a} \leq \frac{1}{2} \frac{v_a}{v_1+v_a} \leq \frac{1}{2} \frac{v_a}{2v_a} = \frac{1}{4} \]
		      hence
		      \[ \frac{v_1v_a}{v_1+v_a} ({\Delta'_a}^2- \Delta_a^2) \geq - \varepsilon  \;. \]

		\item In case $\vmu$ and $\vmu'$ do not share a best arm, define the family of bandits $(\vmu^{(t)})_{t \in [0, 1]}$ by
		      \[ \forall t \in [0, 1], \forall a \in \Segent{K}, \quad \mu^{(t)}_a = (1-t) \mu_a + t \mu'_a  \;. \]
		      One can check that
		      \begin{itemize}
			      \item $\vmu = \vmu^{(0)}$,
			      \item $\vmu' = \vmu^{(1)}$,
			      \item $\norm{\vmu^{(t_1)}-\vmu^{(t_2)}}_\infty \leq \abs{t_1-t_2} \varepsilon$ for every $t_1, t_2 \in [0, 1]$.
		      \end{itemize}
		      Select the subdivision $0 = t_0 < t_1 < \dots < t_N = 1$ of times at which the optimal arms of $\vmu^{(t)}$ are modified. Note that $N \geq 2$ as $\vmu$ and $\vmu'$ do not have a common best arm. Note that by continuity:
		      \begin{itemize}
			      \item for any $n \in \segent{1}{N-1}$, $\vmu^{(t_n)}$ has at least two best arms so that $g(\vmu^{(t_n)}, \vv) = 0$,
			      \item $\vmu^{(1)}$ and $\vmu$ have a common best arm,
			      \item $\vmu^{(N-1)}$ and $\vmu'$ have a common best arm.
		      \end{itemize}
		      Thus
		      \begin{align*}
			      g(\vmu', \vv) - g(\vmu, \vv) & = g(\vmu', \vv) - g(\vmu^{(1)}, \vv) + g(\vmu^{(N-1)}, \vv) - g(\vmu, \vv)        \\
			                                   & \geq - \frac{\norm{\vmu-\vmu^{(1)}}_\infty + \norm{\vmu^{(N-1)}-\vmu'}_\infty}{2} \\
			                                   & \geq - \frac{(t_1 + (1-t_{N-1}) ) \varepsilon}{2} \geq -\frac{\varepsilon}{2} \;.
		      \end{align*}
	\end{itemize}
	\end{proof}

\begin{lemma} \label{lem:moving_w} Let $\vmu' \in \cG$ be a Gaussian bandit and $\vu, \vv \in \Sigma_K$ be such that
	\[ \max_{a \in \Segent{K}} \frac{\abs{u_a-v_a}}{u_a} \leq \eta \]
	for a fixed $0 \leq \eta \leq 1$. Then:
	\[ g(\vmu', \vv) \geq \frac{(1-\eta)^2}{1+\eta} g(\vmu', \vu) \;. \]
\end{lemma}

\begin{proof} Without loss of generality, assume that arm $1$ is one of the best arms of $\vmu'$. Note that the condition of the lemma can be rewritten as
	\[ \forall a \in \Segent{K}, \quad (1-\eta) u_a \leq v_a \leq (1+\eta) u_a \;. \]
	Then for every $a \neq 1$:
	\[ \frac{v_1v_a}{v_1 + v_a} \geq \frac{(1-\eta)^2 u_1 u_a}{(1+\eta) (u_1 + u_a)} \;. \]
	Thus:
	\[ g(\vmu', \vv) = \min_{a \neq 1} \frac{v_1v_a}{v_1 + v_a} {\Delta_a'}^2 \geq  \frac{(1-\eta)^2}{1+\eta} \min_{a \neq 1} \frac{u_1u_a}{u_1 + u_a} {\Delta'_a}^2 = \frac{(1-\eta)^2}{1+\eta} g(\vmu', \vu) \;. \]
	\end{proof}

\begin{proof}[Proof of Proposition~\ref{prop:control_g}] 
	The result follows directly by applying Lemmas~\ref{lem:moving_w} and~\ref{lem:moving_mu} with $\vu = \vw(\vmu)$:
	\[ g(\vmu', \vv) \geq \frac{(1-\eta)^2}{1+\eta} g(\vmu', \vw(\vmu)) \geq \frac{(1-\eta)^2}{1+\eta} \big( g(\vmu, \vw(\vmu)) - \varepsilon/2 \big) \;. \]
	\end{proof}

\section{Proof of the main result}	\label{appendixD} 


The aim of this section is to prove Theorem~\ref{thm:bound}. Let $\gamma \in (0, 1)$ and $\vmu \in \cG^*$. We assume, without loss of generality, that $a^*(\vmu) = 1$. We also write for simplicity $\vDelta = \vDelta(\vmu)$, $\vw = \vw(\vmu)$ and $T = T(\vmu)$.

Recall that the confidence regions are defined, for $t \in \segent{K}{\tau_\delta}$, by
\begin{equation*}
	\textstyle \cC\cR_{\vmu}(t) = \prod_{a \in \Segent{K}} \bigl[\hat{\mu}_a(t) \pm \ell_a(t) \bigr] \;,
\end{equation*}
where $\ell_a(t) = C_{\gamma/K}(N_a(t)) = 2\sqrt{\frac{\log(4KN_a(t)/\gamma)}{N_a(t)}}$. 

Let $\cE$ denotes an event such that $\vmu$ belongs to all confidence regions:
\[ \cE = \bigcap_{t=K}^{\tau_\delta} \big( \vmu \in \cC\cR_{\vmu}(t) \big) \]
and recall that the confidence regions defined by Equation~\eqref{eq:def_CR} are chosen so as to ensure that $\P_{\vmu}(\cE) \geq 1-\gamma$ (see Lemma~\ref{lem:CR_mu}). Furthermore, when $\cE$ occurs, \EBS has been designed so that arms are observed with some minimal linear rate, specified by Lemma~\ref{lem:LB_N} and proved in Appendix~\ref{appendixE1}.
\begin{lemma} \label{lem:LB_N} On event $\cE$ one has:
	\[ \forall t \in \N^*, \quad \min_{a \in \Segent{K}} N_a(t) \geq t w_{\min} - K \; . \]
\end{lemma}
This inequality directly implies the following lower bound:
\begin{equation}	\label{eq:LB_N}
	\forall t \geq \frac{2K}{w_{\min}} , \quad \min_{a \in \Segent{K}} N_a(t) \geq \frac{t w_{\min}}{2} \; .
\end{equation}

\subsection*{Proof outline}

The proof is organized in 3 steps:
\begin{enumerate}
	\item We first show that, on event $\cE$, the optimal vector $\vw$ and the sampling frequency vector $\vN(t)/t$ are very close for any $t \geq T_1$, where $T_1$ is a (problem-dependent) constant. To do so, we will make use of the regularity results of Section~\ref{subsec:regularity} and the fact that the confidence regions shrink with time. 

	\item Then, we control the event $(\tau_\delta > t) \cap \cE$ for $t > T\log(1/\delta)$ by another event for which we can easily bound the probability using Hoeffding's inequality. This inclusion relies once again on the regularity results of Section~\ref{subsec:regularity} and on conditions on $\delta$, in particular we will require to have $T \log(1/\delta) \geq T_1$ with $T_1$ obtained at Step 1.

	\item Finally, we derive the two bounds of the theorem from Hoeffding's inequality and elementary calculations.
\end{enumerate}
The proof uses some technical lemmas introduced and shown in Appendix~\ref{appendixE}.

\subsection*{\underline{Step 1: controlling the difference between vectors $\vw$ and $\vN(t)/t$}}
\textbf{In this step we assume that event $\cE$ occurs. }\\
Let $t \geq \frac{2K}{w_{\min}}$. Equation~\eqref{eq:LB_N} implies that
\[ \forall a \in \Segent{K}, \quad \ell_a(t) = 2\sqrt{\frac{\log(4N_a(t)K/\gamma)}{N_a(t)}} \leq \sqrt{8 \frac{\log(4tK/\gamma)}{t w_{\min}}} =: L(t) \;. \]
$L(t)$ is an arm-independent bound on the half-length of the confidence interval of each $\mu_a$. In other words, $\norm{\tilde{\vmu}(t) - \vmu}_\infty \leq L(t)$ as we are on event $\cE$.
Note that $L(t)$ is deterministic and goes to $0$ as $t$ goes to $+\infty$. This control of $\norm{\tilde{\vmu}(t) - \vmu}_\infty$ together with Theorem~\ref{thm:regularity} allows to control the difference between $\vw$ and $\tilde{\vw}(t)$ for $t$ large enough, as the following Lemma claims.

\begin{lemma} \label{lem:def_T0} Let
	\begin{equation} \label{eq:def_T0}
		T_0 = \max \bigg( \frac{224^2}{\Delta_{\min}^2 w_{\min}} \log \Big( \frac{2\times 224^2 e K}{\Delta_{\min}^2 w_{\min} \gamma} \Big) , \frac{2K}{w_{\min}} \bigg) \;.
	\end{equation}
	Then for every $t \geq T_0$, one has, introducing $\varepsilon_t = \frac{80L(t)}{\Delta_{\min}}$:
	\begin{equation} \label{eq:control_w_wtilde}
		\forall a \in \Segent{K}, \quad w_a (1-\varepsilon_t) \leq \tilde{w}_a(t) \leq w_a (1+\varepsilon_t)\;.
	\end{equation}
\end{lemma}

\begin{proof}
	Let $t \geq \frac{2K}{w_{\min}}$ and assume that $t$ is such that $4L(t) < \Delta_{\min}$. On event $\cE$, one has $\vmu \in \cC \cR_{\vmu}(t) = \prod_{a \in \Segent{K}} [\underline{\mu}_a(t), \overline{\mu}_a(t)]$, hence for any $a \neq 1$:
	\[ \underline{\mu}_1(t) - \overline{\mu}_a(t) \geq \mu_1-2L(t) - (\mu_a+2L(t)) \geq \Delta_a - 4L(t) > 0 \]
	so that the confidence interval for $\mu_1$ is strictly above all other confidence intervals. Hence $\tilde{\vmu}(t)$ has a unique optimal arm which is arm $1$.

	For each arm $a \neq 1$, define $\tilde{\Delta}_a(t) = \Delta_a(\tilde{\vmu}(t)) = \tilde{\mu}_{1}(t) - \tilde{\mu}_a(t)$. Then
	\begin{align*}
		\tilde{\Delta}_a(t)^2                  & \leq (\Delta_a + 2 L(t))^2 = \Delta_a^2 \Big( 1+\frac{4L(t)}{\Delta_a} + \frac{4 L(t)^2}{\Delta_a^2} \Big) \leq \Delta_a^2 \Big( 1+\frac{8L(t)}{\Delta_{\min}} \Big)     \\
		\text{and} \quad	 \tilde{\Delta}_a(t)^2 & \geq (\Delta_a - 2 L(t))^2 = \Delta_a^2 \Big( 1-\frac{4L(t)}{\Delta_a} + \frac{4 L(t)^2}{\Delta_a^2} \Big) \geq \Delta_a^2 \Big( 1-\frac{8L(t)}{\Delta_{\min}} \Big) \;.
	\end{align*}
	If $t$ is such that $\frac{8L(t)}{\Delta_{\min}} \leq 1/7$ (this condition is stronger than $4L(t) < \Delta_{\min}$), we can apply Theorem~\ref{thm:regularity} which gives
	\[ \forall a \in \Segent{K}, \quad w_a (1-\varepsilon_t) \leq \tilde{w}_a(t) \leq w_a (1+\varepsilon_t) \;. \]
	It remains to understand when the condition $\frac{8L(t)}{\Delta_{\min}} \leq 1/7$ holds. We have:
	\[ \frac{8L(t)}{\Delta_{\min}} \leq 1/7 \quad \Longleftrightarrow \quad \frac{\log(4tK/\gamma)}{t} \leq \frac{\Delta_{\min}^2 w_{\min}}{(7 \times 8)^2 \times 8} = \frac{\Delta_{\min}^2 w_{\min}}{2 \times 112^2} \]
	and this inequality is satisfied, by Lemma~\ref{lem:logcard_in}, for
	\[ t \geq \frac{224^2}{\Delta_{\min}^2 w_{\min}} \log \Big( \frac{2\times224^2 e K}{\Delta_{\min}^2 w_{\min} \gamma} \Big) \;. \]
	Combining with the initial condition $t \geq \frac{2K}{w_{\min}}$ leads to the definition of $T_0$.
	\end{proof}

As each $N_a(t)/t$ is nearly the Cesaro sum of the $(\tilde{w}_a(s))_{0 \leq s \leq t-1}$ (see Lemma~\ref{lem:control_N}), and as $\varepsilon_t \to_{t \to +\infty} 0$, we are able to control the difference between $\vw$ and $\vN(t)/t$ after a deterministic time $T_1$.

\begin{lemma} \label{lem:def_T1} Fix $\eta \in (0, 1)$ and let
	\begin{equation} \label{eq:def_T1}
		T_1 = \frac{\max(640^2, 8K)}{\eta^2 \Delta_{\min}^2 w_{\min}^2} \log\Big( \frac{2\times 640^2 e K}{\eta^2 \Delta_{\min}^2 w_{\min}\gamma} \Big)\;.
	\end{equation}
	Then for any $t \geq T_1$ one has:
	\begin{equation} \label{eq:control_w_N}
		\forall a \in \Segent{K}, \quad w_a (1-\eta) \leq \frac{N_a(t)}{t} \leq w_a (1+\eta)\;.
	\end{equation}
\end{lemma}

\begin{proof} Let $T_0$ be defined by Equation~\eqref{eq:def_T0}. Let $t > T_0$ and $a \in \Segent{K}$. Equation~\eqref{eq:control_w_wtilde} of Lemma~\ref{lem:def_T0} gives:
	\[ \abs{\sum_{s=0}^{t-1} \tilde{w}_a(s) - tw_a} \leq \sum_{s=0}^{T_0-1} \abs{\tilde{w}_a(s) - w_a} + \sum_{s=T_0}^{t-1} \abs{\tilde{w}_a(s) - w_a} \leq T_0 + w_a \sum_{s=T_0}^{t-1} \varepsilon_s \;. \]
	By definition of $\varepsilon_t$ one has:
	\begin{align*}
		\sum_{s=T_0}^{t-1} \varepsilon_s & = \frac{80 \sqrt{8}}{\Delta_{\min} \sqrt{w_{\min}}} \sum_{s=T_0}^{t-1} \sqrt{\frac{\log(4sK/\gamma)}{s}}  \leq \frac{80 \sqrt{8} \sqrt{\log(4tK/\gamma)}}{\Delta_{\min} \sqrt{w_{\min}}} \sum_{s=T_0}^{t-1} \frac{1}{\sqrt{s}}  \leq \frac{80 \sqrt{8} \sqrt{t \log(4tK/\gamma)}}{\Delta_{\min} \sqrt{w_{\min}}}
	\end{align*}
	so that we have, using Lemma~\ref{lem:control_N}:
	\begin{align*}
		\abs{\frac{N_a(t)}{t} - w_a} & \leq \frac{1}{t} \Bigg[ \Big|N_a(t) - \sum_{s=0}^{t-1} \tilde{w}_a(s)\Big| + \Big|\sum_{s=0}^{t-1} \tilde{w}_a(s) - w_a\Big| \Bigg] \\
		                               & \leq \frac{K+T_0}{t} + w_a \frac{80\sqrt{8} \sqrt{\log(4tK/\gamma)}}{\Delta_{\min} \sqrt{w_{\min} t}}                             \\
		                               & \leq w_a \Bigg( \frac{K+T_0}{tw_{\min}} + \frac{80\sqrt{8} \sqrt{\log(4tK/\gamma)}}{\Delta_{\min} \sqrt{w_{\min} t}} \Bigg) \;.
	\end{align*}
	Thus the conclusion of the Lemma holds when:
	\[ \max \Big( \frac{K+T_0}{tw_{\min}}, \frac{80\sqrt{8} \sqrt{\log(4tK/\gamma)}}{\Delta_{\min} \sqrt{w_{\min} t}} \Big) \leq \frac{\eta}{2} \]
	and this inequality is satisfied, using Lemma~\ref{lem:logcard_in}, when:
	\[ t \geq \max \Big( \frac{2}{\eta} \frac{K+T_0}{w_{\min}}, \frac{640^2}{\eta^2 \Delta_{\min}^2 w_{\min}} \log\Big( \frac{2 \times 640^2 e K}{\eta^2 \Delta_{\min}^2 w_{\min}\gamma} \Big) \Big) \;.\]
	The definition of $T_0$ implies $K + T_0 \leq \frac{4 \max(112^2, K)}{\Delta_{\min}^2 w_{\min}} \log \Big( \frac{2\times 224^2 e K}{\Delta_{\min}^2 w_{\min}\gamma} \Big)$, hence the inequality still holds for
	\[ t \geq \max \Big( \frac{8 \max(112^2, K)}{\eta \Delta_{\min}^2 w_{\min}^2} \log \Big( \frac{2\times 224^2 e K}{\Delta_{\min}^2 w_{\min}\gamma} \Big), \frac{640^2}{\eta^2 \Delta_{\min}^2 w_{\min}} \log\Big( \frac{2 \times 640^2 e K}{\eta^2 \Delta_{\min}^2 w_{\min}\gamma} \Big) \Big) \]
	and $T_1$ is greater than this lower bound. 
	\end{proof}

\subsection*{\underline{Step 2: a useful inclusion of events}}

We want to control the event $(\tau_\delta > t) \cap \cE$ for $t > T\log(1/\delta)$. For $\delta$ small enough, we have the following inclusion of events.
\begin{lemma} 	\label{lem:events_inclusion} Fix $\eta \in (0, 0.15]$ and let $\delta$ be such that
	\begin{equation} \label{cond:C1} \tag{C1}
		T\log(1/\delta) \geq T_1
	\end{equation}
	where $T_1$ is defined by Equation~\eqref{eq:def_T1} and
	\begin{equation} \label{cond:C2} \tag{C2}
		\log(1/\delta) > \frac{4}{\eta} \log\Big(\frac{8eTR^{1/2}}{\eta}\Big) \;.
	\end{equation}
	Then for any $C \in (0, 1]$:
	\[ \forall t \geq (1+C) \frac{(1+\eta)^2}{(1-\eta)^2} T\log(1/\delta), \quad \big(\tau_\delta > t\big) ~\cap~ \cE \quad \subseteq \quad \Big( \norm{\vmu-\hat{\vmu}(t)}_\infty \geq \frac{C}{T}\Big) ~\cap~ \cE \;. \]
\end{lemma}

\begin{remark} Latter, we will use this Lemma with $C = \frac{1}{\log^{\frac{1}{3}}(1/\delta)}$. \end{remark}

\begin{proof} Assume in the following that $T \log(1/\delta) \geq T_1$ and let $t \geq T \log(1/\delta)$. By definition of $T_1$ and Lemma~\ref{lem:def_T1}, one has
	\begin{equation} \label{eq:control_W_w}
		\max_{a \in \Segent{K}} \abs{\frac{w_a - N_a(t)/t}{w_a}} \leq \eta \;.
	\end{equation}
	Then using Proposition~\ref{prop:control_g} and Equation~\eqref{eq:link_T_g}:
	\begin{align*}
		\big(\tau_\delta > t\big) ~\cap~ \cE \quad & \subseteq \quad \Big(Z(t) = t g(\hat{\vmu}(t), \vN(t)/t) \leq \beta(t, \delta) \Big) ~\cap~ \cE                                                                  \\
		                                           & \subseteq \quad \Big( t \frac{(1-\eta)^2}{1+\eta} \Big( g(\vmu, \vw) - \frac{\norm{\vmu-\hat{\vmu}(t)}_\infty}{2} \Big) \leq \beta(t, \delta) \Big) ~\cap~ \cE \\
		                                           & \subseteq \quad \Big( \frac{\norm{\vmu-\hat{\vmu}(t)}_\infty}{2} \geq \frac{1}{T} - \frac{1+\eta}{(1-\eta)^2} \frac{\beta(t, \delta)}{t}\Big) ~\cap~ \cE \;.
	\end{align*}
	Consider now
	\[ f(t) = \frac{1+\eta}{(1-\eta)^2} \frac{\beta(t, \delta)}{t} = \frac{1+\eta}{(1-\eta)^2} \frac{\log\big( \frac{Rt^\alpha}{\delta} \big)}{t} \;. \]
	As $\alpha \leq 2$, one can check that $f$ is decreasing on $(4, +\infty)$. Let us show that
	\begin{equation} \label{eq:condition_f_Calpha}
		\forall C \in (0, 1], \quad f\Big((1+C)\frac{(1+\eta)^2}{(1-\eta)^2} T\log(1/\delta)\Big) \leq \frac{1}{(1+C)T} \;.
	\end{equation}
	Fix $C \in (0, 1]$. As $\alpha \leq 2$ and as $\eta \leq 0.15$ is such that $\frac{(1+\eta)^2}{(1-\eta)^2} \leq 2$, we have:
	\begin{align*}
		f\Big((1+C)\frac{(1+\eta)^2}{(1-\eta)^2} T\log(1/\delta)\Big) & \leq \frac{1+\eta}{(1-\eta)^2} \frac{\log\Big( \frac{R (4T \log(1/\delta))^2}{\delta} \Big)}{(1+C)\frac{(1+\eta)^2}{(1-\eta)^2} T\log(1/\delta)} \\
		                                                                & \leq \frac{1}{(1+C)T} \frac{1}{1+\eta} \Big( 1 + 2 \frac{\log\big( 4R^{1/2} T \log(1/\delta) \big)}{\log(1/\delta)} \Big) \;.
	\end{align*}
	hence Inequality~\eqref{eq:condition_f_Calpha} is satisfied if
	\[ \log\big( 4R^{1/2} T \log(1/\delta) \big) \leq \frac{\eta}{2} \log(1/\delta) \]
	which is the case, by Lemma~\ref{lem:logcard_in}, when:
	\[ \log(1/\delta) > \frac{4}{\eta} \log\Big(\frac{8eTR^{1/2}}{\eta}\Big) \;. \]
	Finally when Inequality~\eqref{eq:condition_f_Calpha} holds we have for $t \geq (1+C) \frac{(1+\eta)^2}{(1-\eta)^2} T\log(1/\delta)$:
	\begin{align*}
		\big(\tau_\delta > t\big) ~\cap~ \cE \quad & \subseteq \quad \Big( \norm{\vmu-\hat{\vmu}(t)}_\infty \geq \frac{2}{T} - \frac{2}{(1+C)T}\Big) ~\cap~ \cE \\
		                                           & \subseteq \quad \Big( \norm{\vmu-\hat{\vmu}(t)}_\infty \geq \frac{C}{T}\Big) ~\cap~ \cE
	\end{align*}
	where we use $C \leq 1$ in the last inclusion.
	\end{proof}

\subsection*{\underline{Step 3: bounding $\P_{\vmu}\big(\tau_\delta > t ~\cap~ \cE\big)$ and $\E_{\vmu}[\tau_\delta \ind_\cE]$. }}

Fix $\eta \in (0, 1]$ and assume in the following that conditions~\eqref{cond:C1} and~\eqref{cond:C2} of Lemma~\ref{lem:events_inclusion} are satisfied with $\eta' = \eta/7 \leq 0.15$. We set $\zeta = \frac{(1+\eta')^2}{(1-\eta')^2}$. Let $C \in (0, 1]$, $t > (1+C)\zeta T \log(1/\delta)$ and define
\[ \cE_t = \Big( \norm{\vmu-\hat{\vmu}(t)}_\infty \geq \frac{C}{T}\Big) ~\bigcap~ \cE \;. \]

\noindent Lemmas~\ref{lem:events_inclusion} and~\ref{lem:bad_event_control} -- a consequence of Hoeffding's inequality -- (note that Condition~\eqref{cond:C1} ensures that $t \geq \frac{2K}{w_{\min}}$) give the bound:
\begin{equation} \label{eq:bound_tau>t}
	\P_{\vmu}\big(\tau_\delta > t ~\cap~ \cE\big) \leq \P_{\vmu} (\cE_t) \leq 2Kt\exp \Big( - \frac{tw_{\min}}{4{T}^2} C^2 \Big) \;.
\end{equation}
By taking $C = \frac{1}{\log^{\frac{1}{3}}(1/\delta)}$, we obtained so far that
\[ \forall t > \Big(1 + \frac{1}{\log^{\frac{1}{3}}(1/\delta)} \Big) \zeta T \log(1/\delta), \quad \P_{\vmu}\big(\tau_\delta > t ~\cap~ \cE\big) \leq 2Kt\exp \Big( - \frac{tw_{\min}}{4{T}^2} \frac{1}{\log^{\frac{2}{3}}(1/\delta)} \Big) \]
\underline{giving Bound~\eqref{eq:bound_tail}} as long as $\big( 1 + \frac{1}{\log^{\frac{1}{3}}(1/\delta)} \big) \zeta \leq 1+\eta$. Note that $\zeta \leq 1+6\eta'$ as $\eta' \leq 0.15$ so that when
\begin{equation} \label{cond:C3} \tag{C3}
	\frac{1}{\log^{\frac{1}{3}}(1/\delta)} \leq \frac{\eta'}{2} \quad \Longleftrightarrow \quad \log(1/\delta) \geq \frac{8\times 7^3}{\eta^3}
\end{equation}
the condition holds as
\[ \Big( 1 + \frac{1}{\log^{\frac{1}{3}}(1/\delta)} \Big) \zeta \leq \Big(1+\frac{\eta'}{2}\Big)(1+6\eta') \leq 1+6.6\eta' \leq 1+\eta. \]

\underline{It remains to focus on the bound of $\E_{\vmu}[\tau_\delta \ind_\cE]$.} Using Equation~\eqref{eq:bound_tau>t} we have:
\begin{align*}
	\E_{\vmu}[\tau_\delta \ind_\cE] & = \sum_{t = 0}^{\floor{(1+C)\zeta T\log(1/\delta)}} \P_{\vmu}\big(\tau_\delta > t ~\cap~ \cE\big) + \sum_{t > (1+C)\zeta T\log(1/\delta)} \P_{\vmu}\big(\tau_\delta > t ~\cap~ \cE\big) \\
	                                & \leq (1+C)\zeta T\log(1/\delta) + 1 + 2K \sum_{t > (1+C)\zeta T\log(1/\delta)} t\exp \Big( - \frac{tw_{\min}}{4{T}^2} C^2 \Big) \;.
\end{align*}
Define
\[ S(C) = \sum_{t > C \zeta T\log(1/\delta)} t\exp \Big( - \frac{tw_{\min}}{4{T}^2} C^2 \Big) \;. \]
With some technical calculations (see Appendix~\ref{proof:bound_S(C)}), one can obtain that:
\begin{lemma} \label{lem:bound_S(C)} One has
	\[ S(C) \leq \frac{32 {T}^4}{w_{\min}^2} \exp\big( -\frac{w_{\min}}{4T} C^2 \log(1/\delta) \big)  \Big( \frac{\log(1/\delta)}{C^2} + \frac{1}{C^4}\Big) \;. \]
\end{lemma}
Once again, taking $C = \frac{1}{\log^{\frac{1}{3}}(1/\delta)}$ leads to
\begin{align*}
	S(C) & \leq \frac{32 {T}^4}{w_{\min}^2} \exp\Big( -\frac{w_{\min}}{4T} \log^{\frac{1}{3}}(1/\delta) \Big) \Big( \log^{\frac{5}{3}}(1/\delta) + \log^{\frac{4}{3}}(1/\delta) \Big)  \leq \frac{64 {T}^4}{w_{\min}^2} \exp\Big( -\frac{w_{\min}}{4T} \log^{\frac{1}{3}}(1/\delta) \Big) \log^2(1/\delta)
\end{align*}
thus
\[ \E_{\vmu}[\tau_\delta \ind_\cE] \leq \zeta \Big( 1 + \frac{1}{\log^{\frac{1}{3}}(1/\delta)} \Big) T\log(1/\delta) + 1 + \frac{2^7 K {T}^4}{w_{\min}^2} \exp\Big( -\frac{w_{\min}}{4T} \log^{\frac{1}{3}}(1/\delta) \Big) \log^2(1/\delta) \;. \]
Under Condition~\eqref{cond:C3} we get
\[ \zeta \Big( 1 + \frac{1}{\log^{\frac{1}{3}}(1/\delta)} \Big) T\log(1/\delta) + 1 \leq (1+6.6\eta') T \log(1/\delta) + 1 \leq (1+\eta) T \log(1/\delta) \]
\underline{and obtain the Bound~\eqref{eq:bound_E_tau}} claimed in the theorem.
Combining conditions~\eqref{cond:C1},~\eqref{cond:C2} and~\eqref{cond:C3} together, one can define $\delta_0$ satisfying:
\[ \log(1/\delta_0) \geq \frac{7^3 \times \max(2\times 160^2, K)}{\eta^3 \Delta_{\min} w_{\min}^2} \log\Big( \frac{7^2\times 2 \times 640^2 e KR^{1/2}}{\eta^2 \Delta_{\min}^2 w_{\min}\gamma} \Big) \;, \]
with some simplifications allowed by Equation~\eqref{eq:bounds_T} of Proposition~\ref{prop:bounds}.

\section{Technical details for the proof of Appendix~\ref{appendixD}}	\label{appendixE}

\subsection{Proof of Lemma~\ref{lem:LB_N}}	\label{appendixE1}

We will use the following deterministic Lemma:
\begin{lemma} \label{lem:control_N} One has:
	\[ \forall t > 0, \quad \max_{1 \leq a \leq K} \Big| N_a(t) - \sum_{s=0}^{t-1} \tilde{w}_a(s) \Big| \leq K-1 \;. \]
\end{lemma}

\begin{proof} Apply \citet[Lemma 15]{GK16} with $p(s) = \tilde{\vw}(s)$. \end{proof}

The claim is true for $t \in \segent{0}{K}$ as Equation~\eqref{eq:bounds_w1} of Proposition~\ref{prop:bounds} gives
\[ w_{\min}K - K \leq \frac{K}{2} - K \leq 0 \;. \]

Otherwise, fix $t \in \segent{K+1}{\tau_\delta}$ and $a \in \Segent{K}$. For any $s \in \segent{0}{K-1}$, one has $\tilde{w}_a(s) = \frac{1}{K}$ by convention (as all arms are drawn once during the $K$ first rounds, the only request is $\sum_{s=0}^{K-1} \tilde{w}_a(s) = 1$), and thus $\tilde{w}_a(s) \geq w_{\min}$ ($\vw \in \Sigma_K$ implies $w_{\min} \leq \frac{1}{K}$). For any $s \in \segent{K}{\tau_\delta-1}$, one has by Proposition~\ref{prop:optimistic_weights} :
\[ \tilde{w}_a(s) \geq \tilde{w}_{\min}(s) = \max_{\vnu \in \cC \cR_{\vmu}(s)} w_{\min}(\vnu) \geq w_{\min} \]
as $\vmu \in \cC \cR_{\vmu}(s)$ on event $\cE$. Hence by Lemma~\ref{lem:control_N}
\[ N_a(t) \geq \sum_{s=0}^{t-1} \tilde{w}_a(s) - (K-1)  \geq t w_{\min} - (K-1) \geq t w_{\min} - K \;. \]

\subsection{A technical lemma}

\begin{lemma} \label{lem:logcard_in}
	For any $c_1, c_2  > 0$,
	\[ x = \frac{2}{c_1} \log\Big( \frac{c_2 e}{c_1} \Big) \]
	is such that $c_1 x \geq \log(c_2x)$.
\end{lemma}

This is a direct consequence of \citet[Lemma 18]{GK16}.

\subsection{Deviation bound}

We prove the following simple consequence of Hoeffding's inequality.

\begin{lemma} \label{lem:bad_event_control} For any $t \geq \frac{2K}{w_{\min}}$ and $x > 0$, one has
	\[ \P\Big(\max_{a \in \Segent{K}} \abs{\hat{\mu}_a(t) - \mu_a} > x ~\cap~ \cE \Big) \leq 2K t \exp \Big( - \frac{tw_{\min}}{4} x^2 \Big) \;. \]
\end{lemma}

\begin{proof} Fix $t \geq \frac{2K}{w_{\min}}$ and $x > 0$. For any $a \in \Segent{K}$, one has with $T= \frac{tw_{\min}}{2}$:
	\begin{align*}
		\P\Big(\abs{\hat{\mu}_a(t) - \mu_a} > x ~\cap~ \cE \Big) & = \sum_{s=T}^{t} \P\Big(\abs{\hat{\mu}_a(t) - \mu_a} > x ~\cap~ \cE ~\cap ~ N_a(t) = s \Big) &  & \text{by Equation~\eqref{eq:LB_N}}             \\
		                                                         & \leq \sum_{s=T}^{t} \P\Big(\abs{\hat{\mu}_{a, s} - \mu_a} > x \Big)                          &  & \text{by Equation~\eqref{eq:local_estimate}} \\
		                                                         & \leq \sum_{s=T}^{t} 2\exp \Big(-\frac{s}{2}x^2\Big)                                          &  & \text{by Hoeffding's inequality}                    \\
		                                                         & \leq 2t \exp\Big(-\frac{T}{2}x^2\Big)
	\end{align*}
	giving the desired bound by union bound.
	\end{proof}

\subsection{Proof of Lemma~\ref{lem:bound_S(C)}} \label{proof:bound_S(C)}

We have
\[ S(C) = \sum_{t > (1+C) \zeta T\log(1/\delta)} t\exp \Big( - \frac{tw_{\min}}{4{T}^2} C^2 \Big) = \sum_{t > B} f(t) \]
where $f: t \mapsto t\exp (-At)$, $A = \frac{w_{\min}}{4{T}^2} C^2$ and $B = (1+C) \zeta T\log(1/\delta)$. $f$ is increasing until $1/A$ and then decreasing. Let $n_0 = \floor{\frac{1}{A}}$. We will show that $S(C) \leq 2\int_B^{+\infty} f(t) ~\textrm{d}t$. 
\begin{itemize}
	\item If $B > n_0$ then $f$ is decreasing on $[B, +\infty[$ and one has $S(C) \leq \int_B^{+\infty} f(t) ~\textrm{d}t$. 

	\item Otherwise, one has:
	      \begin{align*}
		      S(C) & = \sum_{t = \ceil{B}}^{n_0-1} f(t) + f(n_0) + f(n_0+1) + \sum_{t > n_0+1} f(t)                                                         \\
		           & \leq \sum_{t = \ceil{B}}^{n_0-1} \int_{t}^{t+1} f(t) ~\textrm{d}t + f(n_0) + f(n_0+1) + \sum_{t > n_0+1} \int_{t-1}^t f(t)~\textrm{d}t \\
		           & \leq \int_{\ceil{B}}^{+\infty} f(t)~\textrm{d}t + f(n_0) + f(n_0+1)
	      \end{align*}
	      where in the second inequality, we use the increasing of $f$ on $[B, n_0]$ and its decreasing on $[n_0+1, +\infty]$. The result will be true if
	      \[ f(n_0) + f(n_0+1) \leq \int_{B}^{+\infty} f(t) ~\textrm{d}t \;. \]
	      We have:
	      \begin{align*}
		      f(n_0) + f(n_0+1) & = \floor{\frac{1}{A}} e^{-A\floor{\frac{1}{A}}} + \ceil{\frac{1}{A}} e^{-A\ceil{\frac{1}{A}}}                                                           \\
		                        & \leq \Big( \floor{\frac{1}{A}} + \ceil{\frac{1}{A}} \Big) e^{-A\floor{\frac{1}{A}}}                                                                     \\
		                        & \leq \Big( \floor{\frac{1}{A}} \frac{1}{A} + \frac{1}{A^2} \Big) e^{-A\floor{\frac{1}{A}}}         &  & \text{as } A < \frac{1}{2}                      \\
		                        & = \int_{\floor{\frac{1}{A}}}^{+\infty} f(t) ~\textrm{d}t \leq \int_{B}^{+\infty} f(t) ~\textrm{d}t &  & \text{as } B \leq \floor{\frac{1}{A}} = n_0 \;.
	      \end{align*}
	      where in the last inequality, we used the simple calculation
	      \[ \int_Y^{+\infty} t \exp(-tX) \textrm{d}t = \exp(-YX) \Big( \frac{Y}{X} + \frac{1}{X^2}\Big) \]
	      for $X, Y > 0$.
\end{itemize}
In both cases we have:
\[ S(C) \leq 2\int_{(1+C) \zeta T\log(1/\delta)}^\infty t\exp \Big( - \frac{tw_{\min}}{4{T}^2} C^2 \Big) ~\textrm{d}t \]
and using the same calculation as before
\[ S(C) \leq 2\exp\Big( -\frac{\zeta w_{\min}}{4T} (1+C)C^2 \log(1/\delta) \Big)  \Big( \frac{4(1+C)\zeta {T}^3}{w_{\min}} \frac{\log(1/\delta)}{C^2} + \frac{16{T}^4}{w_{\min}^2} \frac{1}{C^4}\Big) \;. \]
Bounding $C \in (0, 1]$ and $\zeta \in [1, 2]$ (remind that $\zeta \leq 1 + 6\eta')$:
\begin{align*}
	S(C) & \leq 2\exp\Big( -\frac{w_{\min}}{4T} C^2 \log(1/\delta) \Big)  \Big( \frac{16 {T}^3}{w_{\min}} \frac{\log(1/\delta)}{C^2} + \frac{16{T}^4}{w_{\min}^2} \frac{1}{C^4}\Big) \\
	     & \leq \frac{32 {T}^4}{w_{\min}^2} \exp\Big( -\frac{w_{\min}}{4T} C^2 \log(1/\delta) \Big)  \Big( \frac{\log(1/\delta)}{C^2} + \frac{1}{C^4}\Big) \;.
\end{align*}


\section{Proof of asymptotic results}	\label{appendixF}

\subsection{Proof of Lemma~\ref{lem:LB_exploration_rate}}	\label{appendixF1}

We will need the two following lemmas. The first gives a lower bound of $w_{\min}(\vmu)$ and the second provides a lower bound on the minimal gap of the optimistic bandit computed by Algorithm~\ref{algo:OptimisticWeights}. 

\begin{lemma} \label{lem:bound w_min} For any $\vmu \in \cG^*$ one hase $w_{\min}(\vmu) \geq \frac{\Delta_{\min}(\vmu)}{2K}$. 
\end{lemma}

\begin{proof} Let $\vw = \vw(\vmu)$, $w_{\min} = w_{\min}(\vmu)$ and $\vDelta = \vDelta(\vmu)$. We have
\begin{align*}
w_{\min} & = \frac{w_{\max}}{r \Delta_{\max} - 1}		&  & \text{by Equation~\eqref{eq:def_wa} of Proposition~\ref{prop:links_r}} \\
	&\geq \frac{1}{\sqrt{K-1}+1} \times \frac{1}{\frac{\sqrt{K-1}+1}{\Delta_{\min}} \Delta_{\max} - 1} &  & \text{by Inequalities~\eqref{eq:bounds_r} and~\eqref{eq:bounds_w1}}    \\
	 & \geq \frac{\Delta_{\min}}{(\sqrt{K-1}+1)^2}		&  & \text{as } \Delta_{\max}(t) \leq 1                             \\
	 & \geq \frac{\Delta_{\min}}{2K} \;.
	      \end{align*}

\end{proof}

\begin{lemma} \label{lem:min_optimistic_Delta} Let $\cC \cR = \prod_{a \in \Segent{K}} [\underline{\mu}_a, \ov{\mu}_a]$ be a confidence region such that $\underline{\mu}_a < \ov{\mu}_a$ for $a \in \Segent{K}$ and $\max_{a \in \Segent{K}} \underline{\mu}_a = \text{maxLB} > \text{minUB} = \min_{a \in \Segent{K}} \ov{\mu}_a$, and $(\tilde{\vmu}, \vv) \gets \text{\textsc{OptimisticWeights}}(\cC\cR)$. Then
\[ \Delta_{\min}(\tilde{\vmu}) \geq \min_{a \in \Segent{K}} \ov{\mu}_a - \underline{\mu}_a \:. \]
\end{lemma}

\begin{proof} We proceed by contradiction: let us assume that $\tilde{\vmu}$ is such that
	\[ \Delta_{\min}(\tilde{\vmu}) < \min_{a \in \Segent{K}} \ov{\mu}_a - \underline{\mu}_a \:. \]
	By the two hypothesis and the algorithm's procedure, it is clear that $\tilde{\vmu}$ has a unique best arm. Without loss of generality let us arrange the arms so that $\tilde{\mu}_1 > \tilde{\mu}_2 \geq \tilde{\mu}_3 \geq \dots \geq \tilde{\mu}_K$. Note that $\Delta_{\min}(\tilde{\vmu}) = \tilde{\mu}_1-\tilde{\mu}_2$.

	As $1$ is the best arm, once again the algorithm's procedure ensures that $\tilde{\mu}_1 = \ov{\mu}_1$. In addition, our assumption implies $\Delta_{\min}(\tilde{\vmu}) < \ov{\mu}_1 - \underline{\mu}_1$, giving $\tilde{\mu}_2 > \underline{\mu}_1$. Recall that $\tilde{\mu}_2 = \max(\underline{\mu}_2, \text{minUB})$, so that we split our analysis to the two possible cases:
	\begin{itemize}
		\item if $\tilde{\mu}_2 = \underline{\mu}_2$, then we cannot have $\ov{\mu}_2 \leq \ov{\mu}_1 = \tilde{\mu}_1$ otherwise $\Delta_{\min}(\tilde{\vmu}) > \ov{\mu}_2 - \underline{\mu}_2$, which is impossible.

		      Then $\ov{\mu}_2 > \ov{\mu}_1$. By defining $\vnu = (\tilde{\mu}_2, \ov{\mu}_2, \tilde{\mu}_3, \dots, \tilde{\mu}_K)$, one has $\vnu \in \cC \cR$ and $w_{\min}(\vnu) > w_{\min}(\tilde{\vmu})$ by Lemma~\ref{lem:moving_optimal_arm}. Thus $\tilde{\vmu}$ cannot maximize $w_{\min}$ over $\cC \cR$ which is in contradiction with Proposition~\ref{prop:optimistic_weights}.

		\item if $\tilde{\mu}_2 = \text{minUB}$, then $\tilde{\mu}_2 = \tilde{\mu}_3 = \dots = \tilde{\mu}_K$ and thus all confidence intervals share a common point equal to $\tilde{\mu}_2$ (recall that $\tilde{\mu}_2 \in [\underline{\mu}_1, \ov{\mu}_1]$), which is a contradiction with $\text{maxLB} > \text{minUB}$.
	\end{itemize}
\end{proof}

We can now prove Lemma~\ref{lem:LB_exploration_rate}. Let $t \in \segent{0}{\tau_\delta-1}$. We want to lower bound $\tilde{w}_{\min}(t)$.
\begin{itemize}
	\item If at time $t$ one has $\tilde{\vw}(t) = (1/K, \dots, 1/K)$ then $\tilde{w}_{\min}(t) = \frac{1}{K}$.

	\item Otherwise, by construction of Algorithms~\ref{algo:OptimisticWeights} and~\ref{algo:ExplorationBiased} we know that $t \geq K$ and the confidence region $\cC \cR (t)$ is such that at least two confidence intervals are separated. In that case, the optimistic bandit $\tilde{\vmu}(t)$ has a unique optimal arm and Lemma \ref{lem:bound w_min} gives 
	\[ \tilde{w}_{\min}(t) \geq \frac{\tilde{\Delta}_{\min} (t)}{2K} \;. \]
	      One can use Lemma~\ref{lem:min_optimistic_Delta} and note that as $t \geq K$, all arms have already been pulled at least once, hence
	      \[ \tilde{\Delta}_{\min}^{(t)} \geq \min_{a \in \Segent{K}} 2 \ell_a(t) \geq 4 \min_{a \in \Segent{K}} \sqrt{\frac{\log(4N_a(t)K/\gamma)}{N_a(t)}} \geq 4\sqrt{\frac{\log(4K/\gamma)}{t}} \geq 4\sqrt{\frac{\log 8}{t}} \geq \frac{4}{\sqrt{t}} \;. \]

	      Putting everything together one can obtain
	      \[ \tilde{w}_{\min}(t) \geq \frac{2}{K} \frac{1}{\sqrt{t}} \;. \]
\end{itemize}
In both cases we obtained:
\[ \tilde{w}_{\min}(t) \geq \min \Big( \frac{2}{K} \frac{1}{\sqrt{t}}, \frac{1}{K} \Big) \geq \frac{1}{K} \frac{1}{\sqrt{t}} \]
hence for any $a \in \Segent{K}$ and $t \in \N$, we have using Lemma~\ref{lem:control_N}:
\[ N_a(t) \geq \sum_{s=0}^{t-1} \tilde{w}_a(s) - (K-1) \geq \sum_{s=2}^{t-1} \tilde{w}_{\min}(s) - K \geq \frac{1}{K} \sum_{s=2}^{t-1} \frac{1}{\sqrt{s}} - K \geq \frac{1}{K} \int_1^t \frac{1}{\sqrt{s}} \; \textrm{d}s - K \geq \frac{2}{K} \sqrt{t} - K \;. \]

\subsection{Almost sure asymptotic bound}	\label{appendixF2}

\begin{theorem}[Almost sure asymptotic bound] \label{thm:asymptotic_optimality_as} Fix $\gamma \in (0, 1)$, $\alpha \in [1, e/2]$. For any $\vmu \in \cG^*$, Algorithm \EBS with the threshold of Equation~\eqref{eq:threshold} satisfies
\[ \limsup_{\delta \to 0} \frac{\tau_{\delta}}{\log(1/\delta)} \leq \alpha T(\vmu) \quad \P_{\vmu}\text{-a.s.} \;. \]
\end{theorem}

The result was obtained by \citet[Proposition 13]{GK16}. The adaptation to \EBS is straightforward, as soon as we prove the following result. 

\begin{proposition} For any choice of parameters and $\vmu \in \cG^*$, the sampling rule of \EBS satisfies:
	\[ \lim_{t \to +\infty} \hat{\vmu}(t) = \vmu \quad \P_{\vmu}\text{-a.s.} \quad \quad \text{and} \quad \quad \lim_{t \to +\infty} \frac{N(t)}{t} = w(\vmu) \quad \P_{\vmu}\text{-a.s.} \;. \]
\end{proposition}

\begin{proof} Lemma~\ref{lem:LB_exploration_rate} implies that $N_a(t) \to_{t \to +\infty} +\infty$ for all $a \in \Segent{K}$, so that the law of large number gives
	\[ \lim_{t \to +\infty} \hat{\vmu}(t) = \vmu \quad \P_{\vmu}\text{-a.s.} \;.\]
	Remark that for $a \in \Segent{K}$ one has
	\[ \abs{\tilde{\mu}_a(t) - \hat{\mu}_a(t)} \leq C_{\gamma/K}(N_a(t)) = 2\sqrt{\frac{\log(4N_a(t)K/\gamma)}{N_a(t)}} \quad \longrightarrow_{t \to +\infty} \quad 0 \]
	so that we also have
	\[ \lim_{t \to +\infty} \tilde{\vmu}(t) = \vmu \quad \P_{\vmu}\text{-a.s.} \]
	and thus by continuity of function $\vw$ in $\vmu$ (as $\vmu$ has a unique optimal arm):
	\[ \lim_{t \to +\infty} \tilde{\vw}(t) = \vw(\vmu) \quad \P_{\vmu}\text{-a.s.} \;. \]
	Now for all $t \in \N^*$ and $a \in \Segent{K}$ we have:
	\begin{align*}
		\abs{\frac{N_a(t)}{t} - w_a(\vmu)} & \leq \frac{1}{t} \abs{N_a(t) - \sum_{s=0}^{t-1} \tilde{w}_a(s)} + \abs{\frac{1}{t} \sum_{s=0}^{t-1} (\tilde{w}_a(s) - w_a(\vmu))}                                          \\
		                                     & \leq \frac{K-1}{t} + \abs{\frac{1}{t} \sum_{s=0}^{t-1} (\tilde{w}_a(s) - w_a(\vmu))}                                              &  & \text{by Lemma~\ref{lem:control_N}} \\
		                                     & \to_{t \to +\infty} 0
	\end{align*}
	(using the Cesaro Lemma for the second term).
	\end{proof}

\subsection{Proof of Theorem~\ref{thm:asymptotic_optimality_E}}	\label{appendixF3}

Once again this is a direct adaptation of \citet[Theorem 14]{GK16}. Indeed, we can follow the proof as long as the two lemmas shown in this section are satisfied.

Let us recall the notations of \citet{GK16}. We assume that $1$ is the best arm of $\vmu$. Fix $\varepsilon > 0$. By continuity of $\vw$ in $\vmu$, let $\xi \leq \Delta_{\min}(\vmu)/4$ be such that
\[ \max_{\vmu' \in \cI_\varepsilon} \norm{\vw(\vmu')-\vw(\vmu)}_{\infty} \leq \varepsilon \quad \text{where} \quad \cI_\varepsilon = \prod_{a \in \Segent{K}} [\mu_a \pm \xi] \;. \]
Let $T \in \N$ and define $h(T) = T^{1/4}$ and the event
\[ \cE_T = \bigcap_{t=h(T)}^T (\hat{\vmu}(t) \in \cI_\varepsilon) \;. \]

\begin{lemma} There exist two positive constants $B, C$ (that depend on $\vmu$ and $\varepsilon$) such that
	\[ \P_{\vmu}(\cE_T^c) \leq BT \exp(-CT^{1/8}) \;. \]
\end{lemma}

\begin{proof} We have by union bound
	\[ \P_{\vmu}(\cE_T^c) \leq \sum_{t=h(T)}^T \sum_{a \in \Segent{K}} \P_{\vmu}(|\hat{\mu}_a(t) - \mu_a| > \xi) \;. \]
	Then
	\begin{align*}
		\P_{\vmu}(|\hat{\mu}_a(t) - \mu_a| > \xi) & = \sum_{s=\frac{2}{K}\sqrt{t}-K}^t \P_{\vmu}(|\hat{\mu}_a(t) - \mu_a| > \xi ~\cap~ N_a(t) = s) &  & \text{by Lemma~\ref{lem:LB_exploration_rate}} \\
		                                          & \leq \sum_{s=\frac{2}{K}\sqrt{t}-K}^t \P(|\hat{\mu}_{a, s} - \mu_a|> \xi)                      &  & \text{by Equation~\eqref{eq:local_estimate}}  \\
		                                          & \leq 2 \sum_{s=\frac{2}{K}\sqrt{t}-K}^t \exp\Big(-s\frac{\xi^2}{2}\Big)                        &  & \text{by Hoeffding's inequality}                     \\ 
		                                          & \leq 2\frac{\exp(-(\frac{2}{K}\sqrt{t}-K)\xi^2/2)}{1-\exp(-\xi^2/2)} \;.
	\end{align*}
	With
	\[ B = 2 K \frac{\exp(K\xi^2/2)}{1-\exp(-\xi^2/2)} \quad \text{and} \quad C = \frac{\xi^2}{K} \;, \]
	one has
	\[ \P_{\vmu}(\cE_T^c) \leq \sum_{t=h(T)}^T B \exp(-\sqrt{t}C) \leq BT \exp(-\sqrt{h(T)} C) \leq BT \exp(-CT^{1/8}) \;. \]
	\end{proof}

\begin{lemma} There exists a constant $T_{\varepsilon}$ such that for $T \geq T_{\varepsilon}$, il holds that on $\cE_T$
	\[ \forall t \geq \sqrt{T}, \quad \max_{a \in \Segent{K}} \Big| \frac{N_a(t)}{t} - w_a(\vmu) \Big| \leq 3\varepsilon \;. \]
\end{lemma}

\begin{proof} For any $t \geq \sqrt{T} = h(T)^2$ and $a \in \Segent{K}$ we have:
	\begin{align*}
		\abs{\frac{N_a(t)}{t} - w_a(\vmu)} & \leq \frac{1}{t} \abs{N_a(t) - \sum_{s=0}^{t-1} \tilde{w}_a(s)} + \abs{\frac{1}{t} \sum_{s=0}^{t-1} (\tilde{w}_a(s) - w_a(\vmu))}                                          \\
		                                     & \leq \frac{K-1}{t} + \frac{h(T)}{t} + \abs{\frac{1}{t} \sum_{s=h(T)}^{t-1} (\tilde{w}_a(s) - w_a(\vmu))}                          &  & \text{by Lemma~\ref{lem:control_N}} \\
		                                     & \leq \frac{K-1}{T^{1/2}} + \frac{1}{T^{1/4}} + \varepsilon                                                                          &  & \text{by definition of } \cE_T      \\
		                                     & \leq \frac{K}{T^{1/4}} + \varepsilon \leq 3\varepsilon
	\end{align*}
	whenever $T \geq (K/2\varepsilon)^4 = T_{\varepsilon}$.
	\end{proof}


\section{Additional experiments}	\label{appendixG}

In this section we present numerical experiments to compare the dependence on parameter $\delta$ of three strategies, namely \EBS, \TaS and \US (that samples arms uniformly).

On Figure \ref{fig:add_experiment}, we plot for each strategy and several bandit parameters the estimate of $\E_{\vmu}[\tau_\delta]$ for different values of $\delta$ (using the same threshold $\beta$ as in the experiments of Section \ref{sec:experiments} and $\gamma = 0.1$ for \EBS). We also plot in black the lower bound of \cite{GK16} ($\sim_{\delta \to 0} T(\vmu) \log(1/\delta)$).

In term of performance, we observe that \EBS is always between \US and \TaS (which is always quite close to the lower bound). More precisely there are different behaviours: 
\begin{itemize}
	\item when the problem is difficult (with small gaps), \EBS behaves almost like \TaS. Indeed for those parameters the uniform sampling phase of \EBS is relatively small comparing to the required number of samples so that \EBS has time to shrink its confidence regions close to parameter $\vmu$ and thus behaves like \TaS (see bandit $\vmu^{(1)}$), 

	\item when the problem is easier (with large gaps), \EBS behaves like \US, as in almost all simulations the strategy does not have enough confidence to leave the uniform sampling phase before the stopping condition is satisfied (see bandits $\vmu^{(2)}$ and $\vmu^{(3)}$). When $\delta$ decreases, there is a separation between \EBS and \US as more and more simulations reach the non-uniform sampling phase of our strategy. If we continue to check for smaller values of $\delta$, one can expect that \EBS will come closer to \TaS than \US, for the same reasons as before: the confidence regions of \EBS have more time to shrink. This is what we observe we bandit $\vmu^{(4)}$, for which \EBS has the behaviour of \US for moderate values of $\delta$ and then the behaviour of \TaS for small values of $\delta$. 
\end{itemize}

\begin{figure}[h]
  \centering
  \includegraphics[width=0.45\textwidth]{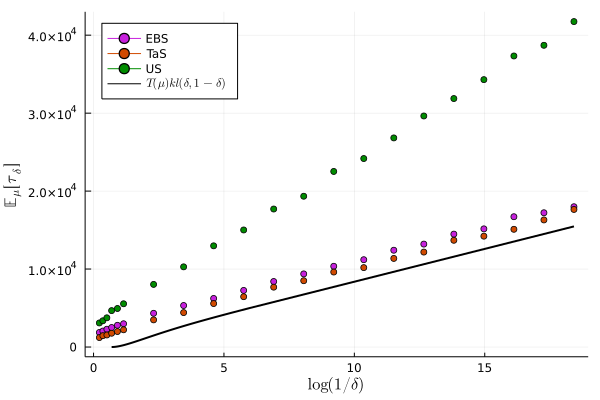} \quad   \includegraphics[width=0.45\textwidth]{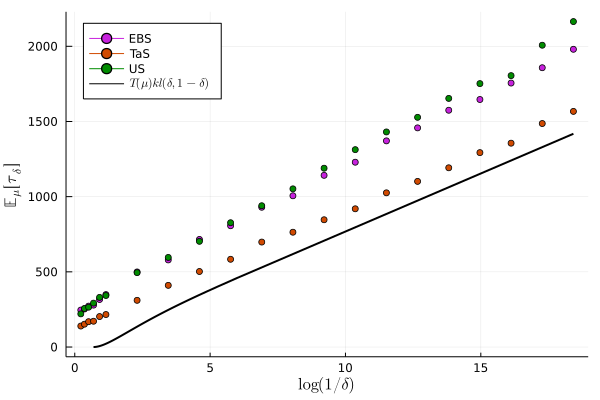} \\
  \includegraphics[width=0.45\textwidth]{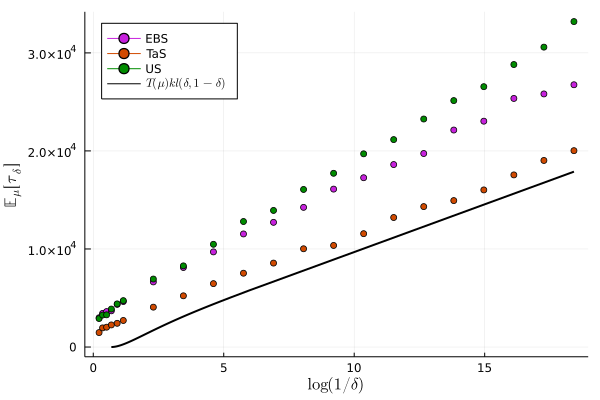}  \quad \includegraphics[width=0.45\textwidth]{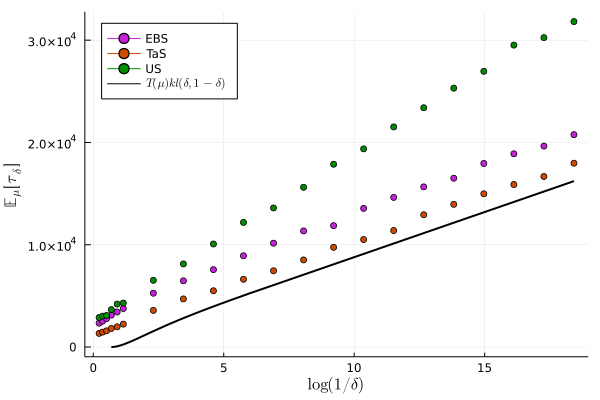} 
  \caption{Empirical Expected Number of Draws $\E_{\vmu}[\tau_\delta]$, Averaged over $500$ Experiments. Top left: $\vmu^{(1)} = (0.9, 0.8, 0.6, 0.4, 0.4)$. Top Right: $\vmu^{(2)} = (0.9, 0.5, 0.45, 0.4)$. Bottom Left: $\vmu^{(3)} = (0.9, 0.8, 0.75, 0.7)$. Bottom Right: $\vmu^{(4)} = (0.9, 0.8, 0.7, 0.6)$}
   \label{fig:add_experiment}
\end{figure}

\end{document}